\RequirePackage{rotating}
\documentclass[reqno]{amsart}

\pdfoutput=1

\usepackage{amssymb}
\usepackage[british]{babel}
\usepackage{booktabs}
\usepackage{color}
\usepackage{enumitem}
\usepackage[T1]{fontenc}
\usepackage[utf8]{inputenc}
\usepackage{mathrsfs}
\usepackage[l2tabu,orthodox]{nag}
\usepackage{tensor}
\usepackage{todonotes}
\usepackage{xypic}
\usepackage[vcentermath]{youngtab}

\usepackage[
		pdftex,
		bookmarks,
		bookmarksopen,
		bookmarksopenlevel=1,
		bookmarksnumbered,
		breaklinks,
		colorlinks,
		linkcolor=linkcolor,
		citecolor=citecolor,
		urlcolor=urlcolor,
	]{hyperref}

\hypersetup{
	pdftitle    = {An Algebraic Geometric Foundation for a Classification of Superintegrable Systems in Arbitrary Dimension},
	pdfauthor   = {Jonathan Kress, Konrad Schöbel, Andreas Vollmer},
	pdfsubject  = {article},
	pdfcreator  = {\LaTeX{} with `amsart' class},
	pdfproducer = {pdf\LaTeX},
	pdfkeywords = {superintegrable systems}
}

\definecolor{linkcolor}{rgb}{0.5,0.0,0.0}
\definecolor{citecolor}{rgb}{0.0,0.5,0.0}
\definecolor{urlcolor} {rgb}{0.0,0.0,0.5}

\title[Classification of Superintegrable Systems in Arbitrary Dimension]
	{An Algebraic Geometric Foundation for a Classification of Superintegrable Systems in Arbitrary Dimension}

\author[J. Kress]{Jonathan Kress}
\email{j.kress@unsw.edu.au}
\address{%
	School of Mathematics and Statistics \\
	The University of New South Wales \\
	Sydney 2052 \\
	Australia
}
\author[K. Schöbel]{Konrad Schöbel}
\email{konrad.schoebel@htwk-leipzig.de}
\address{%
	Faculty for Digital Transformation \\
	HTWK Leipzig University of Applied Sciences \\
	04251 Leipzig \\
	Germany
}
\author[A. Vollmer]{Andreas Vollmer}
\email{andreas.d.vollmer@gmail.com}
\email{andreas.vollmer@polito.it}
\address{%
	School of Mathematics and Statistics \\
	The University of New South Wales \\
	Sydney NSW 2052 \\
	Australia
}
\address{%
	Institute of Geometry and Topology,
	University of Stuttgart,
	70550 Stuttgart, Germany
}
\address{%
	Dipartimento di Scienze Matematiche (DISMA),
	Politecnico di Torino,
	Corso Duca degli Abruzzi, 24,
	10129 Torino, Italy
}

\numberwithin{equation}{section}

\newtheorem{theorem}{Theorem}[section]
\newtheorem{proposition}[theorem]{Proposition}
\newtheorem{lemma}[theorem]{Lemma}
\newtheorem{corollary}[theorem]{Corollary}

\theoremstyle{definition}
\newtheorem{definition}[theorem]{Definition}
\newtheorem{example}[theorem]{Example}
\theoremstyle{remark}
\newtheorem{remark}[theorem]{Remark}

\setlist[enumerate,1]{label=(\roman*)}

\renewcommand{\geq}{\geqslant}

\DeclareMathOperator{\Sym}{Sym}
\begin{document}

\begin{abstract}

	Second-order superintegrable systems in dimensions two and three are
	essentially classified.  With increasing dimension, however, the
	non-linear partial differential equations employed in current methods
	become unmanageable.  Here we propose a new, algebraic-geometric approach
	to the classification problem -- based on a proof that the classification
	space for irreducible non-degenerate second-order superintegrable systems
	is naturally endowed with the structure of a quasi-projective variety with
	a linear isometry action.  On constant curvature manifolds our approach
	leads to a single, simple and explicit algebraic equation defining the
	variety classifying superintegrable Hamiltonians that satisfy all relevant
	integrability conditions generically.  In particular, this includes all
	non-degenerate superintegrable systems known to date and shows that our
	approach is manageable in arbitrary dimension.  Our work establishes the
	foundations for a complete classification of second-order superintegrable
	systems in arbitrary dimension, derived from the geometry of the
	classification space, with many potential applications to related
	structures such as quadratic symmetry algebras and special functions.

\end{abstract}

\maketitle
\tableofcontents

\section{Introduction}

It is a puzzling happenstance that the fundamental equations of nature
admit \emph{explicit analytic} solutions, at least for simple models.  Think
of Schrödinger's Equation for the hydrogen atom, for instance:  Its explicit solutions
describe the shape of the atomic orbits and explain many of the physical and
chemical properties of over hundred elements known today.  The existence of
explicit solutions in this case stems from a deeper fact -- the hydrogen atom is
a \emph{superintegrable system}.  The present paper develops methods to
explore such systems systematically and exhaustively in arbitrary dimension.

\subsection{What are superintegrable systems?}

Symmetries are an essential tool in the study of Hamiltonian systems, and
superintegrable systems are the most symmetric of these.  The prototypical
example of a superintegrable system is the Kepler system of planetary motion
around a central celestial body.  By the nature of the equations of motion,
the movement of the planet is completely determined by its position and
momentum given at any fixed point in time.  More abstractly, the movement
defines a curve in the six-dimensional \emph{phase-space} of position and
momentum.

In a conservative central force field, the energy and the angular momentum
vector are conserved under the temporal evolution of the system.  The Kepler
system has the remarkable property that it possesses an additional conserved
quantity:  the Laplace-Runge-Lenz vector, pointing from the force centre
towards the perihelion of the planetary orbit.  Together, these form seven
scalar constants of motion.  Each of them defines a function on phase-space
and confines the trajectory of the system to a level set of this function.
Since phase space is six-dimensional, only five out of them can be
functionally independent.  Indeed, there are two scalar identities among them.

The quantum counterpart of the Kepler system is the aforementioned model of
the hydrogen atom.  Its conserved quantities are represented by quantum
numbers, which constitute the ordering principle behind the Periodic Table of
Elements.

With the Kepler system in mind, a (maximally) superintegrable system is
defined as a Hamiltonian system of arbitrary dimension~$n$ possessing the maximal number
of functionally independent constants of motion, which is $2n-1$.  The
superintegrable system is called quadratic if the constants of motion can be
chosen quadratic in the momenta.

The study of superintegrable systems has a long standing history due to the
attractive possibility of determining almost all important features using
algebraic methods alone.  Beyond this obvious motivation, worthwhile in its
own right, there exists also a deeper aspiration to classify superintegrable systems:
They give rise to a large class of special functions.

\subsection{Special functions and superintegrable systems}

Since the appearance of the first tables of chords \cite{Almagest}, special
functions have been ubiquitous in science and technology.  Their fundamental
role necessitates not only explicit formulae or the tabulation of a function's
values, but also a thorough documentation of its properties and
interrelations.  Traditionally, this has been done in the form of handbooks,
most notably the Bateman Manuscript Project \cite{Bateman53,Bateman54},
filling five thick volumes, and the ``Abramowitz and Stegun'' Handbook of
Mathematical Functions \cite{A&S}, with more than 40,000 citations one of the
most cited works in the literature \cite{BCLO11}.  In the dawn of the era of
digitisation, the use of sophisticated symbolic computation engines has
overcome the limitations of books and manual calculations, and handbooks have
been replaced by extensive online databases. The most complete resources today
are the Mathematical Functions Site \cite{Wolfram}, which comprises at present
more than 300,000 formulae and is steadily growing, and the NIST Digital
Library of Mathematical Functions \cite{DLMF}, the online version of the above
mentioned Handbook of Mathematical Functions.

Yet, special functions have always been organised in an ad hoc manner and all
handbooks and databases are mere compilations.  Meanwhile, the search for a
unified theory of special functions has continued since the nineteenth century
-- a theory that would explain and systematically organise, for a reasonably
wide class of special functions, their properties, interrelations, symmetry
principles and other related structures behind the façade of seemingly endless
formulae in rows.

A theory aiming to classify special functions may naturally start from some
rich source of such functions.  This is where superintegrable systems come
into play:  Besides the hypergeometric differential equation, they are a
particularly prolific source of special functions.  Notably, it has been shown
that superintegrable systems give rise to hypergeometric orthogonal
polynomials~\cite{KMP07,KMP13}, to Painlevé transcendants
\cite{Marquette_2020,Gravel}, to Jacobi-Dunkl polynomials \cite{GIVZ13}, and to the
recently discovered exceptional polynomials
\cite{Post&Tsujimoto&Vinet,Hoque&Marquette&Post&Zhang}.

\emph{The present work establishes the foundations for a complete
classification of second-order superintegrable systems} and thereby lays the
foundation for subsequent research leading to such a theory -- a ``Periodic
Table of Special Functions'', so to speak,
comprising a wide variety of special
functions derived from a sequence of projective $G$-varieties whose dimensions
and geometric invariants play the role of the atomic and quantum numbers in the
Periodic Table of Elements.  In analogy to the Schrödinger Equation, which
provides the basis for a systematic mathematical description of chemical
elements and their properties, we establish a single, simple equation defining
these varieties, which provides the basis for a systematic algebraic-geometric
description of special functions and their properties.

\subsection{Classification of superintegable systems -- State-of-the-art}
\label{sec:state-of-the-art}

To date, the most well-developed results on the classification and structure
concern second-order (conformally) superintegrable systems on conformally flat
spaces in dimensions two and three. In particular, such results exist for the
Euclidean plane~\cite{Tem04,KKM07b}, Darboux-Koenigs
spaces~\cite{KKW02,KKMW03}, general 2D spaces \cite{DY06,KKM05a,KKM05b},
degenerate 2D superintegrable systems~\cite{KKMP09}, quantum
2D superintegrable systems~\cite{BDK93,KKM06b,DT07}, 3D flat space \cite{KKM07c},
3D conformally flat spaces \cite{KKM05c,KKM06a,KKM07a,KKM07c,CK14} and for
quantum superintegrable systems on 3D conformally flat spaces \cite{KKM06b}.
For an overview see~\cite{Win04}, \cite{MPW13} or the comprehensive
monograph~\cite{KKM18}.  Most of the systems were known prior to their
classification and have been constructed under the additional assumption of
separability or multiseparability~\cite{KMP00a,KMP00b,KK02}.

Above dimension three, only sporadic families of second-order super\-inte\-grable
systems are known, such as the isotropic harmonic oscillator, a generalisation of the
Kepler system~\cite{Plebanski&Przanowski,Ballesteros&Herranz}, respectively
the hydrogen atom~\cite{Nieto}, and the Smorodinsky-Winternitz systems
I~\cite{FMSUW65} and II~\cite{Kalnins&Kress&Miller&Pogosyan}.%
\footnote{%
	We limit ourselves to second-order systems here.  If one includes higher
	order superintegrable systems, additional families are known, such as the
	Calogero-Moser system~\cite{Wojciechowski83} or the Toda
	lattice~\cite{ADS06}.
}
Further $n$-dimensional families can be obtained from these through orbit
degenerations \cite{CKP15}, Stäckel transforms~\cite{BKM86,Post10} or so called Bôcher
contractions \cite{KMSa,KMSb,RKMS,Kalnins&Kress&Miller&Pogosyan}, induced by İnönü-Wigner contractions of
the isometry group \cite{Inonu&Wigner}.

To summarise, to date complete classification results are only known in dimensions two
and three.  Despite the substantial use of computer algebra, an extension of
the classification to higher dimensions is out of the scope of current methods
and therefore one of the most challenging problems in the theory of
superintegrability.  The main reason for this is that the number and the
complexity of the partial differential equations used in current approaches
grows way too fast with the dimension. \emph{In this work we shall overcome
this hindrance and outline a new approach to the classification of second
order superintegrable systems in arbitrary dimension.}

\subsection{What does a ``classification'' actually mean?}

Before one begins to classify superintegrable systems, one should first
clarify what is actually meant by the word ``classification''.  In its
simplest meaning, it stands for an explicit list of all objects under
consideration or, more formally, a bijection with some explicitly given set --
called the \emph{classification space}.  Usually, however, this set carries
much more structure. In the present case of superintegrable systems, for
instance, the classification space can be endowed with at least three natural
structures:
\begin{description}
	\item[Topology]
		As solutions to a system of partial differential equations, the
		classification space inherits a natural topology.
	\item[Group action]
		The definition of superintegrability is invariant under isometries.
		We therefore have a well-defined action of the isometry group $G$ on
		the classification space.
	\item[Equivalence relation]
		Apart from equivalence under isometries, there is a second well known
		transformation for superintegrable systems, called \emph{Stäckel
		equivalence} \cite{BKM86} or \emph{coupling constant metamorphosis} \cite{HGDR84}.
\end{description}
So instead of a bare set, the classification space for superintegrable systems
is at least a topological $G$-space.  This suggests that a classification of
superintegrable systems for a given (pseudo-)Riemannian manifold should be
considered as an isomorphism in the category of topological $G$-spaces, namely
between the classification space and some explicitly given topological
$G$-space.

More generally, one should first fix a category in which to consider the
classification problem for superintegrable systems.  A solution then consists
of the following:
\begin{enumerate}
	\item
		A proof that the classification space is an object in this category.
	\item
		An explicit object in this category.
	\item
		An isomorphism between the classification space and this object.
\end{enumerate}
\emph{Here we will provide the foundations for a classification of second
order superintegrable systems in the category of projective $G$-varieties.}

\subsection{First result: The classification space is a variety}

We prove that the kinetic parts of the constants of motion determine a
non-degenerate second-order superintegrable system up to free constants in the
potential.  Since the kinetic part is given by a Killing tensor, a
non-degenerate superintegrable system on an $n$-dimensional manifold $M$
defines a $(2n-1)$-dimensional subspace in the finite dimensional space
$\mathcal K(M)$ of Killing tensors.  Hence the classification space can
naturally be identified with a subset in the Grassmannian
$G_{2n-1}\bigl(\mathcal K(M)\bigr)$.  We then prove that this subset is
actually a sub\emph{variety}.

Classical theory has always dealt with partial differential equations to solve
the classification problem for superintegrable systems.  Our result now shows
that these equations are, at its heart, purely algebraic equations which come
disguised as partial differential equations in an intricate manner.  This also
indicates that classical techniques are inadequate:  Instead of solving
partial differential equations, one should try to understand the geometry of
the classification space using powerful algebraic-geometric methods, as has
been noticed in the review paper~\cite{MPW13}:
\begin{quote}
	``The possibility of using methods of algebraic geometry to classify
	superintegrable systems is very promising and suggests a method to extend
	the analysis in arbitrary dimension as well as a way to understand the
	geometry underpinning superintegrable systems.''
\end{quote}
Despite the fact that experts in the field agree that an algebraic-geometric
approach is a promising route to a classification of superintegrable systems
in arbitrary dimension, such a route has never been outlined concretely.
\emph{The subject of the present work is to provide exactly this.}

\subsection{State-of-the-art, revised}

In the light of the aforesaid, it should also be mentioned that the explicit
question about the nature of the classification space has never been raised in
the literature.  Most of the currently known classification results for
superintegrable systems mentioned in Section~\ref{sec:state-of-the-art}
consist in writing down lists of normal forms under isometries or Stäckel
transforms \cite{DY06,Kress07}. In other words, they study the \emph{quotient}
of the classification space under these equivalences.  Although never proven
in general, this quotient turns out to be finite in all known cases.
While
passing to the quotient is convenient, as it yields finite lists of simple
normal forms, it destroys most information about the geometry of the
classification space.  The latter is studied implicitly only, by considering
limits of superintegrable systems in the form of orbit degenerations and
Bôcher contractions.  There exists a characterization by polynomial ideals for
non-degenerate systems in dimension~three \cite{Capel} and on flat
two-dimensional space \cite{KKM07b,KKM07c}, but these do have an ad hoc nature
that does not carry over to higher dimensions.

In summary, the currently known classification of second-order superintegrable
systems should be considered a classification in the category of sets, i.e.\
in the most elementary category.  \emph{The results of the present work entail
that the classification problem for superintegrable systems -- in any
dimension -- should be considered in the category of projective
$G$-varieties.}  In this category, the classification problem remains
unsolved, except for non-degenerate systems in the Euclidean plane
\cite{Kress&Schoebel}.

\subsection{Desiderata}

In the present paper, we propose to approach the classification of
superintegrable systems by studying the geometry of the classification space.
Abstractly proving that the classification space is endowed with the structure
of a variety is, however, insufficient, as it does not provide us with
explicit and manageable algebraic equations.  Ideally, for a viable
approach, we desire the equations to have the following properties:

\begin{description}
	\renewcommand{\itemsep}{1ex}
	\item[Explicit]
		The equations should be written down explicitly.
	\item[Concise]
		There should not be too many equations, and they should be simple.
	\item[Generic]
		The equations should have the same form in any dimension,
		except for dimension dependant constants.
	\item[Tensorial]
		The equations should be tensorial, making them independent of
		coordinate changes on the base manifold.
	\item[Equivariant]
		The equations should be explicitly equivariant under isometries.
	\item[Natural]
		The equations should naturally arise from the definition of
		superintegrability and not, e.g., be derived a posteriori from
		a known classification.
	\item[Algebraic]
		The equations should be polynomial.
	\item[Low-degree]
		The algebraic equations should have a low polynomial degree.
	\item[Solvable]
		It should be possible to solve the equations in any dimension, at best
		without resorting to the (excessive) use of computer algebra.
\end{description}

Note that the equations used in the existing literature to classify
superintegrable systems do not satisfy most of these conditions.
\emph{Somewhat surprisingly, however, it turns out that almost all of them can
be satisfied, as we are going to show in the present work.}%
\footnote{%
	In dimension two our equations have a slightly different form, but our
	methods apply as well.  The flat case is treated in \cite{Kress&Schoebel}
	and a general formulation will be subject to a forthcoming paper.
}

\subsection{Second result:  Explicit algebraic equations}

We give explicit algebraic equations for the variety classifying those
superintegrable Hamiltonians on constant curvature manifolds (in dimensions
$n\geq3$) for which all necessary integrability conditions are generically
satisfied.  This variety comprises all non-degenerate superintegrable systems
known to date.  We show that it is isomorphic to the variety of cubic forms
$\Psi_{ijk}x^ix^jx^k$ on $\mathbb R^n$ satisfying the simple algebraic
equation
\begin{equation}
	\label{eq:Master}
	\Psi\indices{^a_{i[j}}\Psi_{k]la}
	=-\kappa g_{i[j}g_{k]l},
\end{equation}
where $\kappa$ is the (constant) sectional curvature and the brackets denote
antisymmetrisation in $j$ and $k$.  Conjecturally, the whole classification
space fibres over this variety.

Furthermore, we show that every superintegrable system in the classification
space gives rise to a torsion-free affine connection which is flat exactly if
the above equation holds.  The origin of this connection lies in the fact that
one can develop a conformally invariant notion of superintegrability for which
conformal equivalences arise from Stäckel transforms.  This suggests that in
the corresponding conformal geometry the Bertrand-Darboux condition gives rise
to tractor bundles equipped with connections parametrised by superintegrable
systems.  We emphasise that classical superintegrability theory, although
dealing with conformally superintegrable systems on conformally flat
manifolds, has never regarded superintegrability from this geometric
perspective.

A reformulation of superintegrability in terms of projective or conformal
geometry is out of the scope of the present publication, as well as a
comprehensive solution of the above equation, a description of the geometry of
the corresponding variety, a derived complete classification of second-order
superintegrable systems on constant curvature manifolds and of related
structures such as quadratic symmetry algebras and hypergeometric orthogonal
polynomials.  \emph{This program will be carried out in future publications,
based on the results in this article.}

\subsection{What can we expect?}

The algebraic-geometric approach employed here to the classification of
\emph{superintegrable} systems is inspired by a similar approach to the
classification of \emph{separable} systems developed by the first author
\cite{Schoebel15,Schoebel16}, which has culminated in a remarkable isomorphism
between the classification space of separable systems (in normal form) on an
$n$-dimensional sphere and the real Deligne-Mumford-Knudsen moduli space
\[
	\bar{\mathscr M}_{0,n+2}(\mathbb R)
\]
of stable genus zero curves with $n+2$ marked points~\cite{Schoebel&Veselov}.
Separable and superintegrable systems are closely related, suggesting that
here as well we may deal with a renowned variety and prominent geometry.

Most known superintegrable systems are multiseparable, meaning that they
contain different separable systems.  This might even be true for all known
superintegrable systems in a broader sense of multiseparability, allowing for
degenerations with multiplicities.  We therefore expect the classification
space for superintegrable systems to be related to symmetric products of
Deligne-Mumford moduli spaces.

The structure of the moduli spaces $\bar{\mathscr M}_{0,n}$ has revealed an
operad structure on the classification spaces of separable systems on spheres,
which provides a simple explicit construction of those systems avoiding
intricate limit procedures \cite{Schoebel&Veselov}.  We expect similar
structures and corresponding constructions for superintegrable systems.

Both classification approaches -- to separable and superintegrable systems --
are contrasted in more detail in Table~\ref{tab:analogy}.

\begin{table}
	\centering
	\begin{tabular}{l|l|l}
		\toprule
		achievement                          & separable                          & superintegrable \\
		                                     & systems                            & systems         \\
		\midrule
		set-theoretical classification       & arbitrary dimension                & dimensions 2, 3 \\
		for constant curvature spaces        & \cite{Kalnins&Miller,Kalnins}      & (see Section~\ref{sec:state-of-the-art}) \\
		\midrule
		proof that the classification space  & \cite{Schoebel16}                  & present paper,                 \\
		is in general an algebraic variety   &                                    & Theorem \ref{thm:main:variety} \\
		\midrule
		explicit algebraic equations         & \cite{Schoebel12}                  & present paper,                   \\
		for constant curvature spaces        &                                    & Equation \eqref{eq:Master} \\
		\midrule
		algebraic-geometric classification   & 3-sphere \cite{Schoebel14}         & Euclidean plane       \\
		for the simplest non-trivial example &                                    & \cite{Kress&Schoebel} \\
		\midrule
		algebraic-geometric classification   & $n$-sphere \cite{Schoebel&Veselov} & forthcoming \\
		in arbitrary dimension               &                                    & paper       \\
		\midrule
		identification of the                & Deligne-Mumford                    & open problem \\
		corresponding algebraic variety      & moduli spaces                      & \\
		\bottomrule
	\end{tabular}
	\bigskip
	\caption{%
		Analogies between the classifications of separable and superintegrable
		systems
	}
	\label{tab:analogy}
\end{table}

\subsection{Perspectives}

Our proposed approach will provide -- in the truest sense of the word -- a
variety of explicit superintegrable systems, i.e.\ Hamiltonian systems that
can be solved exactly by algebraic means.  Apart from this immediate result,
the actual potential of our approach lies in the fact that it transfers the
classification problem for superintegrable systems from the domain of calculus
to that of algebraic geometry, representation theory and geometric invariant
theory, making it accessible to a whole new range of powerful methods.  This
will lead to a series of generalised and induced classifications as well as
universal constructions of many structures related to superintegrable systems,
such as:

\begin{itemize}
	\item degenerate superintegrable systems
	\item conformally superintegrable systems
	\item superintegrable systems on conformal manifolds
	\item multiseparable superintegrable systems
	\item quantum superintegrable systems
	\item quadratic symmetry algebras and their representations
	\item special functions arising from superintegrable systems
\end{itemize}

Let us give an instructive example.  It has been observed that second-order
superintegrable systems in dimension two are in correspondence to
hypergeometric orthogonal polynomials \cite{KMP07,KMP13}.  This correspondence
can probably be formulated properly as an isomorphism in the category of
oriented graphs, with one graph being given by the Askey scheme
\cite{Askey,Askey&Wilson} and the other by a graph whose vertices represent
superintegrable systems and whose edges represent orbit degenerations and
Bôcher contractions.  In our approach the Askey scheme will appear as the
Hasse diagram of the poset of orbits and orbit closure inclusions on the
classification space.  It is interesting to note in this context that a
structure of a glued manifold with corners has been revealed on the Askey
scheme by analysing the limits of hypergeometric orthogonal polynomials
\cite{Koornwinder} and that any variety naturally carries such a structure as
well.  We expect our approach to lead to a proper definition of
higher-dimensional hypergeometric polynomials and to higher dimensional
generalisations of the Askey-Wilson scheme.  Indeed, the generic
superintegrable system on the 3-sphere can be related to 2-variable Wilson
polynomials \cite{KMP11}, and interbasis expansions for the isotropic 3D
harmonic oscillator are linked to bivariate Krawtchouk polynomials
\cite{GVZ14}.  Also, it has been shown that the only free degenerate quadratic
algebras that can be constructed in phase space are those that arise from
superintegrability \cite{EMS17}.

Note that classically, hypergeometric polynomials have always been studied in
families, each parametrised by a few complex parameters.  What we propose here
is a paradigm shift:  Rather than regarding hypergeometric polynomials as
\emph{many} families, each parametrised by a parameter in a \emph{subset of
$\mathbb C^k$}, we propose to describe them as a \emph{single} family,
parametrised by a parameter in a \emph{projective variety}.

\subsection*{Structure of the paper.}

After briefly reviewing theory, terminology and notation in
Section~\ref{sec:preliminaries}, we introduce the pivotal object of our
approach in Section~\ref{sec:T}: A valence three tensor field encoding all
relevant information about a superintegrable system, called the
\emph{structure tensor}.  Sections~\ref{sec:V} and~\ref{sec:K} are devoted to
the integrability conditions that superintegrability imposes on this tensor.
Our first main statement is proven in Section~\ref{sec:variety}, namely that a
properly defined classification space forms a quasi-projective variety.  In
Section~\ref{sec:R=const} we derive explicit algebraic equations for a related
variety on constant curvature spaces.  Finally, in Section~\ref{sec:examples}
the known $n$-dimensional families of superintegrable systems on constant
curvature spaces are reviewed from the point of view developed in this paper.

\subsection*{Acknowledgements.}

We would like to thank Uwe Semmelmann for drawing our
attention towards Codazzi tensors, Rod Gover for his insights into our work
from the viewpoint of parabolic geometry, and Vladimir Matveev for his advice on the integrability of partial differential equations.
Furthermore, we are grateful towards Mike Eastwood, Holger Dullin,
Benjamin McMillan, Paul-Andi Nagy, Joshua Capel and Jeremy Nugent for many
fruitful discussions related to this paper.

The computer algebra system \texttt{cadabra2} \cite{Peeters06,Peeters07} has
been used to find, prove and simplify some of the most important results in
this work.  We would like to express our gratitude to its authors and
contributors for providing, maintaining and extending their software as well
as for distributing it under a free license.

This research was funded by the German Research Foundation (Deutsche
For\-schungsgemeinschaft DFG) -- project number
353063958. A.V.\ acknowledges the postdoc research fellowship received through this grant as well as as a subsequent return fellowship.

This research was funded partially by the Australian Government through the Australian Research Council grant DP190102360.

\section{Preliminaries}
\label{sec:preliminaries}

\subsection{Superintegrable systems}

An $n$-dimensional Hamiltonian system is a dynamical system characterised by a
Hamiltonian function $H(\mathbf p,\mathbf q)$ on the phase space of positions
$\mathbf q=(q_1,\ldots,q_n)$ and momenta $\mathbf p=(p_1,\ldots,p_n)$.  Its
temporal evolution is governed by the equations of motion
\begin{align*}
	\dot{\mathbf p}&=-\frac{\partial H}{\partial\mathbf q}&
	\dot{\mathbf q}&=+\frac{\partial H}{\partial\mathbf p}.
\end{align*}
A function $F(\mathbf p,\mathbf q)$ on the phase space is called a
\emph{constant of motion} or \emph{first integral}, if it is constant under
this evolution, i.e.\ if
\[
	\dot F
	=\frac{\partial F}{\partial\mathbf q}\dot{\mathbf q}
	+\frac{\partial F}{\partial\mathbf p}\dot{\mathbf p}
	=\frac{\partial F}{\partial\mathbf q}\frac{\partial H}{\partial\mathbf p}
	-\frac{\partial F}{\partial\mathbf p}\frac{\partial H}{\partial\mathbf q}
	=0
\]
or
\[
	\{F,H\}=0,
\]
where
\[
	\{F,G\}=
	\sum_{i=1}^n
	\left(
		\frac{\partial F}{\partial q_i}
		\frac{\partial G}{\partial p_i}
		-
		\frac{\partial G}{\partial q_i}
		\frac{\partial F}{\partial p_i}
	\right)
\]
is the canonical Poisson bracket.  Such a constant of motion restricts the
trajectory of the system to a hypersurface in phase space.  If the system
possesses the maximal number of $2n-1$ functionally independent constants of
motion $F^{(0)},\ldots,F^{(2n-2)}$, then its trajectory in phase space is the
(unparametrised) curve given as the intersection of the hypersurfaces
$F^{(\alpha)}(\mathbf p,\mathbf q)=c^{(\alpha)}$, where the constants
$c^{(\alpha)}$ are determined by the initial conditions.  In this case one
can solve the equations of motion exactly and in a purely algebraic way,
without having to solve explicitly any differential equation.

\begin{definition}\label{def:main.concepts}
	\begin{enumerate} 
		\item
			A \emph{maximally superintegrable system} is a Hamiltonian system
			together with a Poisson algebra generated by $2n-1$ functionally
			independent constants of motion~$F^{(\alpha)}$,
			\begin{subequations}
				\begin{align}
					\label{eq:integral}
					\{F^{(\alpha)},H\}&=0,&
					\alpha&=0,1,\ldots,2n-2,
				\end{align}
				one of which is the Hamiltonian itself:
				\begin{equation}
					F^{(0)}=H.
				\end{equation}
			\end{subequations}
		\item 
			A constant of motion is \emph{second-order} if it is of the form
			\begin{subequations}
				\label{eq:quadratic}
				\begin{equation}
					F^{(\alpha)}=K^{(\alpha)}+V^{(\alpha)},
				\end{equation}
				where
				\begin{equation}
					K^{(\alpha)}(\mathbf p,\mathbf q)=\sum_{i=1}^nK^{(\alpha)}_{ij}(\mathbf q)p^ip^j
				\end{equation}
			\end{subequations}
			is quadratic in momenta and $V^{(\alpha)}=V^{(\alpha)}(\mathbf q)$
			is a function depending only on positions.
		\item
			A (maximally) superintegrable system is \emph{second-order} if its constants
			of motion $F^{(\alpha)}$ can be chosen to be second-order
			\begin{subequations}
				\label{eq:Hamiltonian}
				\begin{equation}
					H=g+V,
				\end{equation}
				where
				\begin{equation}
					g(\mathbf p,\mathbf q)=\sum_{i=1}^ng_{ij}(\mathbf q)p^ip^j
				\end{equation}
			\end{subequations}
			is given by the Riemannian metric $g_{ij}(\mathbf q)$ on the underlying
			manifold.
		\item
			We call $V$ a \emph{superintegrable potential} if the Hamiltonian
			\eqref{eq:Hamiltonian} defines a superintegrable system.
	\end{enumerate}
\end{definition}

In this article we will be concerned exclusively with second-order maximally
superintegrable systems and thus omit the terms ``second-order'' and
``maximally'' without further mentioning.

\subsection{Bertrand-Darboux condition}

The condition \eqref{eq:integral} for \eqref{eq:quadratic} and
\eqref{eq:Hamiltonian} splits into two parts, which are cubic respectively
linear in the momenta $\mathbf p$:

\begin{subequations}
	\label{eq:1st+3rd}
	\begin{align}
		\label{eq:3rd}\{K^{(\alpha)},g\}&=0\\
		\label{eq:1st}\{K^{(\alpha)},V\}+\{V^{(\alpha)},g\}&=0
	\end{align}
\end{subequations}

\begin{definition}
	A (second-order) \emph{Killing tensor} is a symmetric tensor field on a
	Riemannian manifold satisfying the Killing equation
	\[
		\{K,g\}=0
	\]
	or, in components,
	\begin{equation}
		\label{eq:Killing}
		K_{ij,k}+K_{jk,i}+K_{ki,j}=0,
	\end{equation}
	where the comma denotes covariant derivatives.
\end{definition}

\begin{example}
	The metric $g$ is trivially a Killing tensor, since it is covariantly
	constant.
\end{example}

The metric $g$ allows us to identify symmetric forms and endomorphisms.
Interpreting a Killing tensor in this way, as an endomorphism on $1$-forms,
Equation \eqref{eq:1st} can be written in the form
\begin{equation}
	\label{eq:potentials}
	dV^{(\alpha)}=K^{(\alpha)}dV,
\end{equation}
and shows that, once the Killing tensors $K^{(\alpha)}$ are known, the
potentials $V^{(\alpha)}$ can be recovered from $V=V^{(0)}$, up to an
irrelevant constant, provided the integrability conditions
\begin{equation}
	\label{eq:dKdV}
	d(K^{(\alpha)}dV)=0
\end{equation}
are satisfied.  This eliminates the potentials $V^{(\alpha)}$ for
$\alpha\not=0$ from our equations.

\subsection{Generalised Cramer's Rule}

The following generalisation of the well-known Cramer's Rule will be used in
order to solve the overdetermined system of linear equations \eqref{eq:dKdV}
for $V$.

\begin{definition}
	The \emph{Gram Coefficients} $G_k(A)$ of a linear map $A$ are
	defined to be the coefficients of the polynomial
	\[
		\det(1+tAA^*)=\sum_{k=0}^\infty G_k(A)t^k,
	\]
	where $A^*$ denotes the adjoint with respect to an inner product.
\end{definition}

Observe that up to sign and order, the Gram Coefficients of $A$ are the
coefficients of the characteristic polynomial of $AA^*$.  In particular,
$G_k(A)$ is homogeneous of degree $2k$ for real $A$.  The following result is
a consequence of the Cayley-Hamilton Theorem.

\begin{proposition}\cite{DTGVL}
	\label{prop:Moore-Penrose}
	A linear map~$A$ on an inner product space has rank~$r$ if and only if
	\begin{equation}
		\label{eq:rank}
		G_r(A)\not=0=G_{r+1}(A).
	\end{equation}
	In this case, the system of linear equations
	\[
		Ax=b
	\]
	has a solution $x$ if and only if
	\[
		G_{r+1}(A|b)=0.
	\]
	Moreover, the minimal norm solution is given by
	\[
		x=A^\dagger b,
	\]
	where
	\begin{equation}
		\label{eq:Moore-Penrose}
		A^\dagger=\frac1{G_r(A)}\sum_{k=1}^rG_{r-k}(A)(-A^*A)^{k-1}A^*.
	\end{equation}
	is the \emph{Moore-Penrose inverse} of $A$.
\end{proposition}

\subsection{Young projectors}

We will make extensive use of Young projectors, mainly to make tensor
symmetries explicit and to simplify lengthy tensor expressions.  Since here is
not the place for a comprehensive introduction to the representation theory of
symmetric and linear groups, we refer to the literature on this subject, e.g.\
\cite{Fulton,Fulton&Harris} and content ourselves with providing only those examples
appearing in the present work.

A \emph{partition} of a positive integer $n$ is a decomposition of $n$ into a
sum of ordered positive integers:
\begin{align*}
	n&=\lambda_1+\lambda_2+\cdots+\lambda_r&
	\lambda_i&\in\mathbb N&
	\lambda_1\geq\lambda_2\geq\cdots\geq\lambda_r&>0.
\end{align*}
A \emph{Young frame} is a visualisation of a partition by consecutive,
left-aligned rows of square boxes, such as
\[
	\text{\tiny\yng(4,2,2,1)}
	\qquad\text{for}\qquad
	9=4+2+2+1.
\]
Young frames are used to label irreducible representations of the permutation
group~$S_n$ and the induced Weyl representations of $\mathrm{GL}(n)$.  A
\emph{Young tableau} is a Young frame filled with distinct objects, in our
case tensor index names.  Young tableaux are used to define explicit
projectors onto irreducible representations. Let us illustrate this with a couple of
examples used in this article.

A Young tableau consisting of a single row is used to denote complete
symmetrisation, as in
\[
	\young(ijk)S_{ijk}
	=S_{ijk}
	+S_{ikj}
	+S_{kij}
	+S_{kji}
	+S_{jki}
	+S_{jik}.
\]
Similarly, a single column Young tableau denotes complete antisymmetrisation,
\[
	\young(i,j,k)A_{ijk}
	=A_{ijk}
	-A_{ikj}
	+A_{kij}
	-A_{kji}
	+A_{jki}
	-A_{jik}.
\]
A general Young tableau denotes the composition of its row symmetrisers and
column antisymmetrisers.  By convention, we apply antisymmetrisers first.
Operators of this type are (scalar multiples of) projectors, called
\emph{Young projectors}.  The Young projectors used most here are \emph{hook
symmetrisers}, composed of a single row and a single column.  For instance,
\[
	\young(ji,k)
	T_{ijk}
	=
	\young(ji)
	\young(j,k)
	T_{ijk}
	=
	\young(ji)
	(T_{ijk}-T_{ikj})
	=T_{ijk}-T_{ikj}
	+T_{jik}-T_{jki}.
\]
If we want to apply the symmetrisers first, we can use the adjoint operator.
For example
\[
	{\young(ji,k)}^*
	T_{ijk}
	=
	\young(j,k)
	\young(ji)
	T_{ijk}
	=
	\young(j,k)
	(T_{ijk}+T_{jik})
	=T_{ijk}+T_{jik}
	-T_{ikj}-T_{kij}.
\]
Next, it is easy to see that tensors of the form
\[
	R_{ijkl}
	={\young(ij,kl)}^*T_{ijkl}
	=
	\young(i,k)
	\young(j,l)
	\young(ij)
	\young(kl)
	T_{ijkl}
\]
are algebraic curvature tensors, i.e.\ satisfy
\begin{enumerate}
	\item antisymmetry: $R_{jikl}=-R_{ijkl}$,
	\item pair symmetry: $R_{klij}=R_{ijkl}$,
	\item the Bianchi identity: $R_{ijkl}+R_{iklj}+R_{iljk}=0$.
\end{enumerate}
We will use a subscript ``$\circ$'' to indicate a projector onto the
completely trace-free part.  For example,
\[
	W_{ijkl}
	=\frac1{12}{\young(ij,kl)}^*_\circ R_{ijkl}
\]
is the Weyl part in the well known Ricci decomposition
\[
	R_{ijkl}
	=W_{ijkl}
	+\frac1{4 (n-1)}{\young(ik,jl)}^*\mathring R_{ik}g_{jl}
	+\frac1{8n(n-1)}{\young(ik,jl)}^*g_{ik}g_{jl},
\]
where
\[
	\mathring R_{ij}
	=\frac12\,{\young(ij)}_\circ\,R_{ij}
	=R_{ij}-\frac Rng_{ij}
\]
is the trace-free part of the Ricci tensor and $R$ the scalar curvature.

We will also use Young tableaux to denote symmetrisations in a subset of a
tensor's indices, such as in
\[
	\young(j,k)T_{ijk}=T_{ijk}-T_{ikj}.
\]

\section{The structure tensor of a superintegrable system}
\label{sec:T}

Let $M$ be a connected Riemannian manifold of dimension $n\geq3$ with metric
$g$ and Levi-Civita connection $\nabla$.  For simplicity we will -- here and
in what follows -- denote covariant derivatives with a comma and the
trace-free part of the Hessian of $V$ by
\begin{equation}
	\label{eq:Hessian:tracefree}
	\mathring V_{,ij}=V_{,ij}-\frac1n\Delta Vg_{ij}.
\end{equation}
Then, in components, the Bertrand-Darboux condition \eqref{eq:dKdV} for a
Killing tensor $K$ in a superintegrable system reads

\begin{equation}
	\label{eq:dKdV:ij}
	\young(i,j)
	\bigl(K\indices{^m_i}\mathring V_{,jm}+K\indices{^m_{i,j}}V_{,m}\bigr)=0.
\end{equation}

We consider this equation for $K=K^{(\alpha)}$ with $\alpha=0,1,\ldots,2(n-1)$
as a linear system
\begin{equation}
	\label{eq:Ax=b}
	Ax=b,
\end{equation}
where the vector $x$ contains the unknown components of the trace-free Hessian
\eqref{eq:Hessian:tracefree}, the coefficient matrix $A$ the components of the
Killing tensors $K^{(\alpha)}$ and the right hand side $b$ the components of
the second term in the sum \eqref{eq:dKdV:ij} for each $K^{(\alpha)}$.

If the Killing tensors are analytic, the components of the coefficient matrix
$A$ and hence the Gram coefficients $G_k(A)$ are analytic as well.  In
particular, on a Riemannian manifold $M$ with analytic metric, the rank of $A$
is constant on an open and dense subset of $M$ by
Proposition~\ref{prop:Moore-Penrose}.

\begin{definition}
	We say a superintegrable system on a Riemannian manifold $M$ has
	\emph{rank} $r$, if the rank of the coefficient matrix $A$ in
	\eqref{eq:Ax=b} has rank $r$ on an open and dense subset of $M$.
\end{definition}

Note that the maximal rank of a superintegrable system is
\begin{equation}
	\label{eq:rank:max}
	r_\text{max}=\frac{n(n+1)}2-1=\frac{(n-1)(n+2)}2.
\end{equation}
A maximal rank superintegrable system can be characterised more explicitly in
terms of its Killing tensors as follows.  Recall that the Riemannian metric on
the base manifold provides an isomorphism between bilinear forms and
endomorphisms on the tangent space, so that we can identify both silently.

\begin{definition}\label{def:irreducible}
	\begin{enumerate}
		\item
			A set of endomorphisms is \emph{irreducible} if they do not have a
			non-trivial invariant subspace in common.
		\item
			A set of endomorphism fields on a Riemannian manifold $M$ is
			called \emph{irreducible}, if they are pointwise irreducible
			on an open and dense subset of~$M$.
		\item
			We call a superintegrable system \emph{irreducible}, if its
			Killing tensors form an irreducible set.
	\end{enumerate}
\end{definition}

\begin{lemma}
	\label{lemma:irreducible}
	A superintegrable system has maximal rank if and only if it is irreducible.
\end{lemma}

\begin{proof}
	Observe that the first term in the sum \eqref{eq:dKdV:ij} can be written
	as a commutator $[K,\mathring V'']$ of endomorphisms, where $\mathring
	V''$ denotes the trace-free part of the Hessian of~$V$.  The kernel of the
	coefficient matrix $A$ in \eqref{eq:Ax=b} therefore consists of all
	trace-free symmetric endomorphisms commuting with all Killing tensors
	$K^{(\alpha)}$ in the superintegrable system.  Since $A$ has more rows
	than columns, it has maximal rank if and only if its kernel is trivial.
	By Schur's Lemma this is the case if and only if the Killing tensors form
	an irreducible set.
\end{proof}

As a consequence of Proposition~\ref{prop:Moore-Penrose}, we get:

\begin{proposition}
	\label{prop:Wilczynski}
	Every irreducible superintegrable system on a Riemannian manifold $M$
	admits a tensor field $T$ with the following properties:
	\begin{enumerate}
		\item
			$T$ is well-defined and smooth on an open and dense subset of $M$.
		\item
			$T$ has degree three and is symmetric and trace-free in its first
			two indices:
			\begin{align}
				\label{eq:T:symmetries}
				T_{jik}&=T_{ijk}&
				g^{ij}T_{ijk}&=0
			\end{align}
		\item
			The superintegrable potential satisfies
			\begin{equation}
				\label{eq:Wilczynski}
				V_{,ij}
				=T\indices{_{ij}^m}V_{,m}
				+\frac1ng_{ij}\Delta V.
			\end{equation}
		\item
			$T$ is uniquely determined by the Killing tensors $K^{(\alpha)}$
			in the superintegrable system.
		\item
			$T$ only depends on the subspace spanned by the Killing tensors
			$K^{(\alpha)}$, i.e. it is invariant under linear basis changes
			\[
				K^{(\alpha)}
				\mapsto
				\sum_\beta c_{\alpha\beta}K^{(\beta)}
				\qquad
				c\in\mathrm{GL}(n)
			\]
	\end{enumerate}
	The components $T_{ijk}$ of $T$ are given explicitly in terms of the
	Killing tensors by the rank $r$ Moore-Penrose inverse, where
	$r=r_\text{max}$ is the maximal rank \eqref{eq:rank:max}, and are
	well-defined over the complement of the set $\{G_r(A)=0\}$.
\end{proposition}

We remark that equations similar to \eqref{eq:Wilczynski} appear in
\cite{KKM05c}, in local coordinates and for dimension three.

\begin{definition}
	We call the tensor $T_{ijk}$ in Proposition~\ref{prop:Wilczynski} the
	\emph{structure tensor} of an irreducible superintegrable system.
\end{definition}

\begin{example}
	\label{ex:IHO}
	The \emph{isotropic harmonic oscillator} on flat $n$-space has a vanishing
	structure tensor.  It is an irreducible system in the sense of
	Definition~\ref{def:irreducible} and has the potentials
	\[
		V(\mathbf x)=\frac{\omega^2}2(\mathbf x-\mathbf x_0)^2+V_0
	\]
	with $n+2$ free parameters $\omega^2$, $\mathbf x_0$ and $V_0$ as solutions to
	\eqref{eq:Wilczynski}. Note that $V$ can be linearly parametrised by setting $\mathbf{a}=-\omega^2\mathbf{x}_0$ and $a_0=\frac12\omega^2\mathbf{x}_0^2+V_0$.
\end{example}

\begin{example}
	The special case $\mathbf x_0=\mathbf 0$ is compatible with the squares of
	the angular momenta
	\[
		K^{(ij)}=(x^idx^j-x^jdx^i)^2.
	\]
	In dimension $n>3$ these define a \emph{non-maximal} superintegrable
	system which is reducible. Indeed, one easily verifies that
	$K^{(ij)}dV=0$, confirming that $dV$ is a common eigenvector of the
	$K^{(ij)}$.
\end{example}

We would like to mention that our methods are inspired by Wilczynski's series
of papers on the projective differential geometry of surfaces
\cite{Wilczynski_I,Wilczynski_IV}.

\section{Superintegrable potentials}
\label{sec:V}

\subsection{Prolongation of a superintegrable potential}
\label{subsec:V:prolongation}

Equation~\eqref{eq:Wilczynski} expresses the derivative of $\nabla V$ linearly
in $\nabla V$ and $\Delta V$, with coefficients that are determined by the
structure tensor.  The following Proposition shows that this equation can be
extended by a second one to a system expressing the derivatives of $\nabla V$
and $\Delta V$ both linearly in $\nabla V$ and $\Delta V$, with the
coefficients determined by the structure tensor.  An extension of such type is
called \emph{prolongation}.

\begin{proposition}
	\label{prop:prolongation:V}
	The potential of a superintegrable system with structure tensor $T_{ijk}$
	satisfies
	\begin{subequations}
		\label{eq:prolongation:V}
		\begin{alignat}{9}
			\label{eq:prolongation:V:1}
			V_{,ij}
			&=T\indices{_{ij}^m}&&V_{,m}
			+\tfrac1n&g_{ij}&\Delta V\\
			\label{eq:prolongation:V:2}
			\tfrac{n-1}n(\Delta V)_{,k}
			&=q\indices{_k^m}&&V_{,m}
			+\tfrac1n&t_k&\Delta V,
		\end{alignat}
	\end{subequations}
	with the definitions
	\begin{subequations}
		\begin{align}
			\label{eq:t}
			t_j&:=T\indices{_{ij}^i}\\
			\label{eq:q}
			q\indices{_j^m}&:=Q\indices{_{ij}^{im}},
		\end{align}
	\end{subequations}
	where
	\begin{equation}
		\label{eq:Q}
		Q\indices{_{ijk}^m}:=
		T\indices{_{ij}^m_{,k}}
		+T\indices{_{ij}^l}T\indices{_{lk}^m}
		-R\indices{_{ijk}^m}.
	\end{equation}
\end{proposition}

\begin{proof}
	Equation~\eqref{eq:prolongation:V:1} is a copy of
	Equation~\eqref{eq:Wilczynski}.  Substituting it into its covariant
	derivative, we obtain
	\begin{align*}
		V_{,ijk}&
		=T\indices{_{ij}^m_{,k}}V_{,m}
		+T\indices{_{ij}^m}V_{,mk}
		+\tfrac1ng_{ij}(\Delta V)_{,k}\\&
		=\bigl(
			T\indices{_{ij}^m_{,k}}+T\indices{_{ij}^l}T\indices{_{lk}^m}
		\bigr)V_{,m}
		+\tfrac1n
		\bigl(
			T_{ijk}\Delta V
			+g_{ij}(\Delta V)_{,k}
		\bigr).
	\end{align*}
	Antisymmetrisation in $(j,k)$ and application of the Ricci identity yields
	\[
		R\indices{^m_{ijk}}V_{,m}
		=\young(j,k)
		\Bigl[
			\bigl(
				T\indices{_{ij}^m_{,k}}
				+T\indices{_{ij}^l}T\indices{_{lk}^m}
			\bigr)V_{,m}
			+\tfrac1n
			\bigl(
				T_{ijk}\Delta V
				+g_{ij}(\Delta V)_{,k}
			\bigr)
		\Bigr].
	\]
	Solving for the last term on the right hand side, we get
	\[
		\frac1n\young(j,k)g_{ij}(\Delta V)_{,k}
		=-\young(j,k)
		\bigl(
			Q\indices{_{ijk}^m}V_{,m}
			+\tfrac1nT_{ijk}\Delta V
		\bigr).
	\]
	The contraction of this equation in $(i,j)$ now yields
	\eqref{eq:prolongation:V:2}, since $T_{ijk}$ and $Q\indices{_{ijk}^m}$ are
	trace-free in $(i,j)$ by definition.
\end{proof}

The System~\eqref{eq:prolongation:V} can be used to express all higher
derivatives of $\nabla V$ and $\Delta V$ linearly in $\nabla V$ and $\Delta
V$.  In particular, all higher derivatives of $V$ in a fixed point are
determined by the values of $V$, $\nabla V$ and $\Delta V$ in that point.  So
if $V$ is analytic, this determines $V$ locally up to a constant.  This
remains true even if $V$ is not analytic.  Therefore, the space of solutions
of the initial partial differential equation~\eqref{eq:Wilczynski} is finite
dimensional with maximal dimension $n+2$.  This motivates the following
generalisation of the notion of non-degeneracy commonly employed in dimensions
two and three~\cite{Kalnins&Kress&Pogosyan&Miller}.

\begin{definition}
	\label{def:non-degenerate}
	We call a superintegrable system \emph{non-degenerate}, if Equation
	\eqref{eq:Wilczynski} admits an $(n+2)$-dimensional space of solutions
	$V$.
\end{definition}

\subsection{Integrability conditions for a superintegrable potential}

Non-degen\-eracy is just the condition that assures that the integrability
conditions of the System~\eqref{eq:prolongation:V} are satisfied generically,
i.e. independently of the potential.  This will eliminate the potential $V$,
leaving equations involving only the structure tensor, respectively the
Killing tensors of the superintegrable system.

\begin{proposition}
	\label{prop:SIC:V}
	The following are necessary and sufficient conditions for the existence
	and uniqueness of a solution $V$ of the prolongation equation
	\eqref{eq:prolongation:V}, given the values of $\nabla V$ and $\Delta V$
	in a fixed point $x_0\in M$:
	\begin{subequations}
		\label{eq:SIC:V}
		\begin{align}
			\label{eq:SIC:V:linear}
			\young(j,k)\Bigl(T_{ijk}+\tfrac1{n-1}g_{ij}t_k\Bigr)&=0\\
			\label{eq:SIC:V:quadratic}
			\young(j,k)\Bigl(Q_{ijkl}+\tfrac1{n-1}g_{ij}q_{kl}\Bigr)&=0\\
			\label{eq:SIC:V:cubic}
			\young(k,l)
			\bigl(
				q\indices{_k^n_{,l}}
				+T\indices{_{ml}^n}q\indices{_k^m}
				+\tfrac1{n-1}t_kq\indices{_l^n}
			\bigr)&=0.
		\end{align}
	\end{subequations}
\end{proposition}

\begin{proof}
	The system \eqref{eq:prolongation:V} allows us to write all higher
	derivatives of $\nabla V$ and $\Delta V$ as linear combinations of $\nabla
	V$ and $\Delta V$.  Necessary and sufficient integrability conditions are
	then obtained by applying this procedure to the left hand sides of the
	Ricci identities
	\begin{align*}
		\young(j,k)V_{,ijk}&=R\indices{^m_{ijk}}V_{,m}&
		\young(k,l)(\Delta V)_{,kl}&=0.
	\end{align*}
	This results in
	\begin{align*}
		\young(j,k)
		\bigl(
			Q\indices{_{ijk}^m}
			+\tfrac1{n-1}g_{ij}q\indices{_k^m}
		\bigr)V_{,m}
		+\frac1n\young(j,k)
		\bigl(
			T_{ijk}
			+\tfrac1{n-1}g_{ij}t_k
		\bigr)\Delta V
		&=0\\
	\intertext{and, respectively,}
		\young(k,l)
		\bigl(
			q\indices{_k^n_{,l}}
			+T\indices{_{ml}^n}q\indices{_k^m}
			+\tfrac1{n-1}t_kq\indices{_l^n}
		\bigr)V_{,n}
		+\frac1n\young(k,l)
		\bigl(
			t_{k,l}
			+q_{kl}
		\bigr)\Delta V
		&=0.
	\end{align*}
	For a non-degenerate superintegrable potential the coefficients of $\Delta
	V$ and $\nabla V$ must vanish.  In addition to the stated integrability
	conditions, this yields the condition
	\begin{equation}
		\label{eq:t+q}
		\young(k,l)
		\bigl(
			t_{k,l}
			+q_{kl}
		\bigr)=0.
	\end{equation}
	The latter is redundant, however, as it can be obtained from
	\eqref{eq:SIC:V:quadratic} via a contraction over $(i,l)$.
\end{proof}

In the remainder of this section we cast the above integrability coditions for
a superintegrable potential into the following simpler form.

\begin{proposition}
	\label{prop:SIC:V:bis}
	In dimension $n\geq3$, the integrability conditions~\eqref{eq:SIC:V} for a superintegrable
	potential are equivalent to the algebraic conditions
	\begin{subequations}
		\label{eq:SIC:V:bis}
		\begin{align}
			\label{eq:SIC:V:symmetries}
			{\young(ji,k)}^*_{\mspace{-28mu}\circ\mspace{+28mu}}T_{ijk}&=0\\
			\label{eq:SIC:V:Weyl}
			\frac18{\young(ik,jl)}^*_\circ T\indices{^a_{ik}}T_{ajl}&=W_{ijkl}
		\end{align}
		and the differential condition
		\begin{equation}
			\label{eq:SIC:V:differential}
			{\young(kji,l)}\,\left(T_{ijk,l}+\frac2{n-2}g_{ik}Z_{jl}\right)=0
		\end{equation}
		with
		\begin{equation}
			\label{eq:Z}
			Z_{ij}:=
			\mathring T\indices{_i^{ab}}\mathring T_{jab}
			-(n-2)(\mathring T\indices{_{ij}^a}\bar t_a+\bar t_i\bar t_j)
			-R_{ij},
		\end{equation}
		where $\mathring T_{ijk}$ and $\bar t_i$ are the trace-free part and
		the rescaled non-vanishing trace of the structure tensor, given in
		\eqref{eq:T:components}.
	\end{subequations}
\end{proposition}

\begin{proof}
	Proposition~\ref{prop:SIC:V:bis} follows from
	Propositions~\ref{prop:SIC:V:linear}, \ref{prop:SIC:V:quadratic} and
	\ref{prop:SIC:V:cubic} below.
\end{proof}

\subsubsection{The 1\textsuperscript{st} integrability condition}

We can solve Equation~\eqref{eq:SIC:V:linear} right away, because it is linear
and does not involve derivatives.

\begin{proposition}
	\label{prop:SIC:V:linear}
	The first of the integrability conditions~\eqref{eq:SIC:V} can be written
	in the form~\eqref{eq:SIC:V:symmetries} and is equivalent to the following
	decomposition of the structure tensor:
	\begin{equation}
		\label{eq:St2T}
		T_{ijk}
		=\mathring T_{ijk}
		+\young(ij)\Bigl(\bar t_ig_{jk}-\frac 1ng_{ij}\bar t_k\Bigr),
	\end{equation}
	where
	\begin{subequations}
		\label{eq:T:components}
		\begin{align}
			\mathring T_{ijk}&=\frac16{\young(ijk)\!}_\circ\,T_{ijk}\\
			\bar t_i&=\frac n{(n+2)(n-1)}T\indices{_{ij}^j}
		\end{align}
	\end{subequations}
	are the trace-free part and the rescaled non-vanishing trace of the
	structure tensor.  Note that both are uniquely determined by $T_{ijk}$ and
	vice versa.
\end{proposition}

\begin{proof}
	First note that \eqref{eq:SIC:V:linear} can be written in the form
	\eqref{eq:SIC:V:symmetries} by the definition of the Young projector.

	The structure tensor $T_{ijk}$ is symmetric in $(i,j)$.  According to the
	Littlewood-Richardson rule its symmetry class is therefore
	\[
		\yng(2)\otimes\yng(1)
		\cong
		\yng(3)\oplus\yng(2,1)
		\cong
		{\yng(3)}_\circ\oplus\yng(1)\oplus{\yng(2,1)}_\circ\oplus\yng(1).
	\]
	This means that $T_{ijk}$ can be decomposed into a trace-free and totally
	symmetric part, a trace-free part of hook symmetry, and two independent
	traces.  By \eqref{eq:SIC:V:symmetries} the trace-free hook symmetric part
	vanishes.  This implies that the trace-free part $\mathring T_{ijk}$ of
	the structure tensor is totally symmetric.  The two remaining trace terms
	are of the form
	\begin{align*}
		&{\young(ijk) }\,g_{ij}\sigma_k&
		&{\young(ji,k)}\,g_{ij}\tau_k
	\end{align*}
	and their traces $\sigma_k$ respectively $\tau_k$ can be determined from
	\eqref{eq:T:symmetries} and \eqref{eq:t}.
\end{proof}

\begin{corollary}~
	\label{cor:symmetry}
	\begin{enumerate}
		\item The tensor $q_{ij}$ is symmetric: $q_{ji}=q_{ij}$.
		\item The tensor $t_i$ is the derivative of a function $t$,
			i.e.\ $t_i=t_{,i}$, and similarly for $\bar t_i$.
	\end{enumerate}
\end{corollary}

\begin{proof}
	The first statement follows from substituting \eqref{eq:St2T} into the
	definition \eqref{eq:q} of $q_{ij}$.  The second then follows from
	\eqref{eq:t+q}.
\end{proof}

\subsubsection{The 2\textsuperscript{nd} integrability condition}

\begin{proposition}
	\label{prop:SIC:V:quadratic}
	Assume the structure tensor of a superintegrable system satisfies the
	first integrability condition~\eqref{eq:SIC:V:linear}.  Then the second
	integrability condition~\eqref{eq:SIC:V:quadratic} is equivalent to
	\eqref{eq:SIC:V:Weyl} and \eqref{eq:SIC:V:differential}.
\end{proposition}

\begin{proof}
	Equation~\eqref{eq:SIC:V:Weyl} follows directly by
	symmetrising~\eqref{eq:SIC:V:quadratic} in $(i,l)$.
	Antisymmetrising instead, we obtain
	\begin{equation}
		\label{eq:SIC:V:Ricci:bis}
		\young(ij)\young(k,l)
		\left(
			T_{ijk,l}
			+\frac2{n-2}
			g_{ik}Z_{jl}
		\right)
		=0
	\end{equation}
	with $Z_{ij}$ given by \eqref{eq:Z}.  The symmetriser can be written as
	\[
		\frac12\young(ij)\cdot\frac12\young(k,l)
		=
		\frac1{96}
		{\young(kji,l)}^*
		 \young(kji,l)
		+
		\frac1{96}
		 \young(ji,k,l)
		{\young(ji,k,l)}^*,
	\]
	making explicit the projectors in the Littlewood-Richardson rule
	\[
		\yng(2)\otimes\yng(1,1)
		\cong
		\yng(3,1)\oplus\yng(2,1,1).
	\]
	The last component vanishes, because
	\[
		{\young(ji,k,l)}^*\!\!
		\left(
			T_{ijk,l}
			+\frac2{n-2}
			g_{ik}Z_{jl}
		\right)
		=2
		\young(j,k,l)
		\left(
			-g_{ij}t_{kl}
			+\frac1{n-2}
			(
				 g_{ik}Z_{jl}
				+g_{jk}Z_{il}
			)
		\right)
		=0,
	\]
	where we have used \eqref{eq:SIC:V:linear}.  The equivalence of
	\eqref{eq:SIC:V:differential} and \eqref{eq:SIC:V:Ricci:bis} now follows
	from the fact that $PP^*P$ is proportional to $P$ for any Young projector~$P$.
\end{proof}

\begin{lemma}
	The integrability condition~\eqref{eq:SIC:V:differential} implies
	\begin{equation}
		\label{eq:SIC:integrability}
		\young(ij)\young(k,l,m)
		\left(
			{R^a}_{ilm}T_{ajk}
			+\frac2{n-2}g_{ik}Z_{jl,m}
		\right)
		= 0.
	\end{equation}
\end{lemma}

\begin{proof}
	Differentiation and antisymmetrisation of \eqref{eq:SIC:V:Ricci:bis} yields
	\[
		\young(ij)
		\young(k,l,m)
		\left(
			T_{ijk,lm}
			+\frac2{n-2}g_{ik}Z_{jl,m}
		\right)
		=0.
	\]
	The statement now follows from applying the Ricci identity to the first
	term and using the Bianchi identity.
\end{proof}

\subsubsection{The 3\textsuperscript{rd} integrability condition}

\begin{proposition}
	\label{prop:SIC:V:cubic}
	The third of the integrability conditions~\eqref{eq:SIC:V} is redundant.
\end{proposition}

\begin{proof}
	A straightforward computation confirms that, up to trace terms,
	Equation~\eqref{eq:SIC:V:cubic} is a linear combination of the contraction
	of~\eqref{eq:SIC:V:differential} and~\eqref{eq:SIC:integrability}, if we
	take into account the Ricci identity for $t_k$. Similarly, the trace
	of~\eqref{eq:SIC:V:cubic} is a linear combination
	of~\eqref{eq:SIC:V:differential}, contracted with $\mathring T_{ijk}$,
	again taking into account the Ricci identity for~$t_k$.  This proves the
	claim.
\end{proof}

\section{Superintegrable Killing tensors}
\label{sec:K}

\subsection{Prolongation of a superintegrable Killing tensor}
\label{sec:prolongation.Killing.tensors}

For arbitrary second-order Killing tensors $K_{ij}$ it is well known that
all higher covariant derivatives are determined by the derivatives up to
second order.  More precisely, an explicit but complicated expression for
$K_{ij,klm}$ can be given which is linear in $K_{ij}$ and $K_{ij,k}$, the
coefficients being linear in the Riemannian curvature tensor and its
derivative \cite{Wolf98}, see also \cite{Gover&Leistner}.  Symbolically,
\begin{equation}
	\label{eq:prolongation:K:full}
	\nabla^3K=(\nabla R)\boxtimes K+R\boxtimes(\nabla K),
\end{equation}
where ``$\boxtimes$'' is a placeholder for some complicated bilinear
operation.  This defines the \emph{standard prolongation} of the Killing
equation.

The following proposition shows that a Killing tensor which arises from a
superintegrable system satisfies another, much simpler prolongation, namely
that all its covariant derivatives are already determined by the Killing
tensor itself. Symbolically:
\[
	\nabla K=T\boxtimes K.
\]
This generalizes, to arbitrary dimension, equations found in \cite{KKM05c} for
dimension three.

\begin{proposition}
	\label{prop:prolongation:K}
	A Killing tensor in a non-degenerate superintegrable system with structure
	tensor $T_{ijk}$ satisfies
	\begin{equation}
		\label{eq:prolongation:K}
		K_{ij,k}=\frac13\young(ji,k)T\indices{^a_{ji}}K_{ak}.
	\end{equation}
\end{proposition}

\begin{proof}
	Substituting \eqref{eq:Wilczynski} into \eqref{eq:dKdV:ij} gives
	\[
		\young(j,k)
		\bigl(
			K\indices{^a_{j,k}}
			-T\indices{_{jb}^a}K\indices{^b_k}
		\bigr)V_{,a}=0.
	\]
	From the definition of non-degeneracy it then follows that
	\begin{equation}
		\label{eq:shortcut:1,1}
		\young(j,k)K_{ij,k}=
		\young(j,k)T\indices{^a_{ji}}K_{ak}.
	\end{equation}
	On the other hand, the Killing equation \eqref{eq:Killing} implies that
	\[
		K_{ij,k}=\frac13\young(ji,k)K_{ij,k}.
	\]
	Combining the last two equations proves~\eqref{eq:prolongation:K}.
\end{proof}

\begin{lemma}
	Any Killing tensor satisfies the following identity
	\begin{equation}
		\label{eq:noname}
		\young(il)\young(j,k)
		\left(
			K_{ij,kl}
			+\young(jl)R\indices{^a_{ijk}}K_{al}
		\right)
		=0.
	\end{equation}
\end{lemma}

\begin{proof}
	Using the identities
	\begin{align*}
		\young(jk)K_{ij,k}&=-\frac12\young(jk)K_{jk,i}&
		\young(j,k)R\indices{^a_{jki}}&=-\frac12\young(j,k)R\indices{^a_{ijk}},
	\end{align*}
	which follow from a symmetrisation of the Killing equation respectively an
	antisymmetrisation of the Bianchi identity, we have
	\begin{align*}
		\young(il)\young(j,k)
		K_{ij,kl}
		&=
		\young(il)\young(j,k)
		\left(
			K_{ij,lk}
			+\young(k,l)K_{ij,kl}
		\right)\\
		&=
		\young(il)\young(j,k)
		\left(
			-\frac12K_{il,jk}
			+\young(ij)R\indices{^a_{ikl}}K_{aj}
		\right)\\
		&=
		\young(il)\young(j,k)
		\left(
			-\frac12R\indices{^a_{ijk}}K_{al}
			+R\indices{^a_{ikl}}K_{aj}
			+R\indices{^a_{jkl}}K_{ai}
		\right)\\
		&=
		\young(il)\young(j,k)
		\left(
			 R\indices{^a_{ikl}}K_{aj}
			-R\indices{^a_{ijk}}K_{al}
		\right)
	\end{align*}
\end{proof}

\begin{lemma}
	\label{lemma:pointwise}
	Suppose the values of the Killing tensors and of the structure tensors of
	two non-degenerate superintegrable systems coincide in a fixed point.
	Then the values of the covariant derivatives of the structure tensors also
	coincide in this point.
\end{lemma}

\begin{proof}
	Substituting the derivative of \eqref{eq:shortcut:1,1},
	\[
		\young(j,k)K_{ij,kl}
		=
		\young(j,k)
		\left(
			 T\indices{^a_{ji,l}}K_{ak  }
			+T\indices{^a_{ji  }}K_{ak,l}
		\right),
	\]
	together with \eqref{eq:prolongation:K} into \eqref{eq:noname} yields
	\begin{equation}
		\label{eqn:konrads.gleichung}
		\young(il)\young(j,k)
		\left(
			T\indices{^a_{ji,l}}K_{ak}
			+\frac13T\indices{^a_{ji}}\young(ka,l)T\indices{^b_{ka}}K_{bl}
			+\young(jl)R\indices{^a_{ijk}}K_{al}
		\right)
		=0.
	\end{equation}
	Now suppose we have two structure tensors with the same values in a fixed
	point~$x_0$.  Denote their difference by $\delta T_{ijk}$.  Then the
	difference of the two copies of the above equation at~$x_0$, obtained for
	each of the structure tensors, is
	\[
		\young(il)\young(j,k)
		\delta T\indices{^a_{ji,l}}(x_0)K_{ak}(x_0)
		=0.
	\]
	This equation is satisfied by all Killing tensors in the superintegrable
	system.  Observe that, for fixed $i$ and $l$, the left hand side is a
	commutator between $\delta T\indices{^a_{ji,l}}(x_0)$
	and~$K\indices{^a_k}(x_0)$.  Hence, if the system is irreducible, we have
	\[
		\young(il)\delta T\indices{^a_{ji,l}}(x_0)=g^a_j\Lambda_{il}
	\]
	for some symmetric tensor $\Lambda_{il}$.  Contracting $a$ and $j$ shows that
	$\Lambda_{il}=0$ and hence
	\[
		\young(il)\delta T\indices{^a_{ji,l}}(x_0)=0.
	\]
	On the other hand, by~\eqref{eq:SIC:V:Ricci:bis},
	\[
		\young(i,l)\delta T\indices{^a_{ji,l}}(x_0)=0,
	\]
	implying $\delta T\indices{^a_{ji,l}}(x_0)=0$.  This shows that both
	structure tensors must have the same derivatives at $x_0$.
\end{proof}

\begin{proposition}
	\label{prop:uniqueness}
	Suppose the values of the Killing tensors and of the structure tensors of
	two non-degenerate superintegrable systems coincide in a fixed point.
	Then the two systems have the same Killing tensors and the same structure
	tensor.
\end{proposition}

\begin{proof}
	Under the hypothesis of the proposition we conclude that, at the fixed
	point, there also coincide the following: the values of the first
	derivatives of the Killing tensors by \eqref{eq:prolongation:K}, of the
	derivatives of the structure tensor by the preceding lemma and of the
	second derivatives of the Killing tensors by the identity
	\begin{equation}
		\label{eq:K''}
		K_{ij,kl}
		=
		\frac13\young(ji,k)
		\left(
		T\indices{^a_{ji,l}}K_{ak}
		+\frac13\young(kb,l)T\indices{^b_{ji}}T\indices{^a_{kb}}K_{al}
		\right),
	\end{equation}
	obtained from substituting \eqref{eq:prolongation:K} into its own
	derivative.  Since the values of a Killing tensor and its first and second
	derivatives in a single point uniquely determine this tensor in a
	neighbourhood of this point, the Killing tensors of both systems coincide.
	By definition, their structure tensors then coincide as well.
\end{proof}

\subsection{Integrability conditions for a superintegrable Killing tensor}

The integrability conditions for the standard prolongation of a Killing
tensor, Equation~\eqref{eq:prolongation:K:full}, are
\begin{multline*}
	K_{ij,kl}-K_{kl,ij}
	=\tfrac12\young(ij)\young(kl)\young(ik)K_{im}R\indices{^m_{klj}}\\
	+\young(ij)K_{im}R\indices{^m_{jkl}}
	-\young(kl)K_{km}R\indices{^m_{lij}}
\end{multline*}
and a very lengthy expression of the form
\[
	(\nabla^2R+R\boxtimes R)\boxtimes K+
	(\nabla R)\boxtimes(\nabla K)+
	R\boxtimes(\nabla^2K)
	=0,
\]
involving (even in low dimension) several hundreds of terms \cite{Wolf98}, see
also \cite{Gover&Leistner}.  In contrast, the following proposition shows that
the integrability conditions for the prolongation of a superintegrable Killing
tensor, Equation~\eqref{eq:prolongation:K}, are much simpler.

\begin{proposition}
	\label{prop:SIC:K}
	The following are necessary and sufficient conditions for the existence
	and uniqueness of a solution $K_{ij}$ of the prolongation equation
	\eqref{eq:prolongation:K}, given the values of $K_{ij}$ in a fixed point
	$x_0\in M$:
	\begin{equation}
		\label{eq:SIC:K}
		\young(k,l)
		\Bigl(
			P\indices{_{ijk}^{ab}_{,l}}
			+P\indices{_{ijk}^{pq}}P\indices{_{pql}^{ab}}
			-
			\frac12
			\young(ij)
			g^a_iR\indices{^b_{jkl}}
		\Bigr)
		K_{ab}
		=0,
	\end{equation}
	where
	\[
		P\indices{_{ijk}^{ab}}
		:=\frac16\young(ab)\young(ji,k)g^a_kT\indices{^b_{ji}}.
	\]
\end{proposition}

\begin{proof}
	Equation \eqref{eq:prolongation:K} can be used to express all higher
	derivatives of the Killing tensor linearly in the Killing tensor itself.
	Explicitly, writing \eqref{eq:prolongation:K} as
	\begin{equation}
		\label{eq:prolongation:K:bis}
		K_{ij,k}=P\indices{_{ijk}^{ab}}K_{ab}
	\end{equation}
	and substituting it back into its own derivative, yields
	\[
		K_{ij,kl}
		=P\indices{_{ijk}^{ab}_{,l}}K_{ab}
		+P\indices{_{ijk}^{cd}}K_{cd,l}
		=
		\left(
			 P\indices{_{ijk}^{ab}_{,l}}
			 +P\indices{_{ijk}^{cd}}P\indices{_{cdl}^{ab}}
		\right)
		K_{ab}.
	\]
	This expression must satisfy the Ricci identity
	\[
		\young(k,l)K_{ij,kl}=\young(ij)R\indices{^a_{ikl}}K_{aj},
	\]
	which is the integrability condition for the existence of a local solution
	to \eqref{eq:prolongation:K}.
\end{proof}

The following definition plays the same role for superintegrable Killing
tensors as that of non-degeneracy for superintegrable potentials:  It assures
that the integrability conditions~\eqref{eq:SIC:K} are generically satisfied,
that is independently of the Killing tensors.

\begin{definition}
	We call a non-degenerate superintegrable system \emph{abundant} if
	Equation~\eqref{eq:prolongation:K} has $n(n+1)/2$ linearly independent
	solutions.
\end{definition}

The following is a specification of Proposition~\ref{prop:SIC:K} for abundant
systems.

\begin{corollary}
	\label{cor:abundant}
	The following are necessary and sufficient conditions that the space of
	solutions $K_{ij}$ to the prolongation equation \eqref{eq:prolongation:K}
	assumes the maximal dimension $n(n+1)/2$:
	\begin{equation}
		\label{eq:SIC:K:generic}
		\young(mn)
		\young(k,l)
		\Bigl(
			P\indices{_{ijk}^{mn}_{,l}}
			+P\indices{_{ijk}^{pq}}P\indices{_{pql}^{mn}}
			-
			\frac12
			\young(ij)
			g^m_iR\indices{^n_{jkl}}
		\Bigr)
		=0
	\end{equation}
	\qed
\end{corollary}

\begin{remark}
	In dimension two, every superintegrable system is trivially abundant,
	since $2n-1=n(n+1)/2$ for $n=2$.  For $n=3$ we have $2n-1=5$ and
	$n(n+1)/2=6$.  The so-called ``$5\Rightarrow6$ Lemma'' states that every
	non-degenerate second-order maximally superintegrable system on a
	conformally flat manifold of dimension three is abundant \cite{KKM05c}.
\end{remark}

\subsection{Non-linear prolongation of the structure tensor}

\begin{proposition}
	\label{prop:prolongation:T}
	The generic integrability conditions \eqref{eq:SIC:K:generic} for an
	abundant superintegrable system are equivalent, in dimensions $n\geq3$,
	to the following polynomial expressions for the derivatives of the
	structure tensor,
	\begin{subequations}
		\label{eq:prolongation:T}
		\begin{align}
			T_{ijk,l}
				&=\mathring T_{ijk,l}
				+\young(ij)\Bigl(\bar t_{i,l}g_{jk}-\frac 1ng_{ij}\bar t_{k,l}\Bigr)\\
			\label{eq:DT:DS}
			\mathring{T}_{ijk,l}
				&=\frac1{18}{\young(ijk)}_\circ
				\bigg[
					\mathring{T}\indices{_{ij}^a}\mathring{T}_{kla}
					+ \mathring{T}_{ijk}\bar t_l
					+3\,\mathring{T}_{ijl}\bar t_k
					\notag \\
				&\qquad\qquad\qquad\qquad +
					\left(
						\frac4{n-2}\,\mathring{T}\indices{_i^{ab}}\mathring{T}_{jab}
						-3\,\mathring{T}\indices{_{ij}^a}\bar t_a
					\right)g_{kl}
				\bigg]\\
			\label{eq:DT:Dt}
			\bar t_{k,l}
				&=\frac13
				  \left(
					-\frac2{n-2}\mathring{T}\indices{_k^{ab}}\mathring{T}_{lab}
					+3\mathring{T}\indices{_{kl}^a}\bar t_a
					+4\bar t_k\bar t_l
				  \right)_\circ
				  \notag \\
				&\qquad
					+\frac1ng_{kl}\,\left(
						\frac{3n+2}{6(n+2)(n-1)}\,\mathring{T}^{abc}\mathring{T}_{abc}
						-\frac{n-2}6\,\bar t^a\bar t_a
						+\frac{3}{2(n-1)}\,R
					\right)
		\end{align}
	\end{subequations}
	together with the polynomial equations
	\begin{subequations}
		\label{eq:SIC:K:1st}
		\begin{align}
			\label{eq:SIC:K:Weyl}
			\frac18{\young(ik,jl)}^*_\circ T\indices{^a_{ik}}T_{ajl}
			&=W_{ijkl}=0\\
			\label{eq:SIC:K:Ricci}
			-\frac14\mathring Z_{ij}
			&=\mathring R_{ij}\,.
		\end{align}
	\end{subequations}
	Here ``$\circ$'' denotes the trace-free part, $W_{ijkl}$ and $R_{ij}$ are
	the Weyl respectively the Ricci tensor of the Riemannian manifold and
	$Z_{ij}$ is defined in \eqref{eq:Z}.
\end{proposition}

\begin{proof}
	The generic integrability conditions \eqref{eq:SIC:K:generic} are linear
	in the derivatives of the structure tensor.  They can be solved for these
	derivatives, which yields \eqref{eq:prolongation:T}.  Substituting
	\eqref{eq:prolongation:T} back into \eqref{eq:SIC:K:generic} yields
	\eqref{eq:SIC:K:1st}.
\end{proof}

\begin{corollary}
	Abundant superintegrable systems can only exist on Weyl flat manifolds.
\end{corollary}

\begin{proof}
	In dimension $n=2$ any manifold is Weyl flat. For dimensions $n\geq3$, the
	claim follows from the theorem above.
\end{proof}

\begin{corollary}
	\label{cor:SIC:K->V}
	The generic integrability conditions~\eqref{eq:SIC:K:generic} for a
	superintegrable Killing tensor imply the integrability
	conditions~\eqref{eq:SIC:V} for a superintegrable potential.
\end{corollary}

\begin{proof}
	This follows from Proposition~\ref{prop:prolongation:V} after the
	substitution of~\eqref{eq:prolongation:T}
	into~\eqref{eq:SIC:V:differential}.
\end{proof}

\subsection{Integrability conditions for the structure tensor}

We now change our point of view and consider \eqref{eq:prolongation:T} as a
system of partial differential equations for a tensor $T_{ijk}$ and ask for
necessary and sufficient conditions such a tensor has to satisfy in a single
point, in order to extend to the structure tensor of a superintegrable system.

Contrary to the prolongations~\eqref{eq:prolongation:V} for a superintegrable
potential and \eqref{eq:prolongation:K} for a Killing tensor in a
superintegrable system, the prolongation~\eqref{eq:prolongation:T} for the
structure tensor is \emph{non-linear} \cite{KLV86,Goldschmidt1967,BCG+13}:
Derivatives of $T_{ijk}$ are expressed as quadratic polynomials in~$T_{ijk}$.
In all three cases the prolongation equations allow to express all higher
derivatives of the prolonged tensors polynomially in these tensors.  But the
second derivatives are not independent.  They have to satisfy the Ricci
identity, which becomes a polynomial condition.  While this first order
condition is sufficient for the integrability of the prolongation equation in
the linear case, higher order integrability conditions arise in the non-linear
case:  Taking the covariant derivative of any polynomial condition and
replacing derivatives by the prolongation equation gives a new polynomial
condition.  This procedure terminates once the newly obtained condition lies
in the ideal generated algebraically by all the preceding ones.  The latter
then form a full set of (finitely many) algebraic integrability conditions for
the prolongation.%
\footnote{
	The same procedure was applied for flat two-dimensional superintegrable
	systems in \cite{KKM07b}.
}

\begin{remark}
	\label{rem:procedure}
	The situation we have here is slightly more general, as we seek a solution
	to the non-linear prolongation equation~\eqref{eq:prolongation:T} that
	also satisfies additional algebraic conditions~\eqref{eq:SIC:K:1st} which
	do not originate from the Ricci identity a priori.  Note that the latter
	are of degree two in the structure tensor while the first order
	integrability conditions are of third degree.  This can be achieved by
	adjoining them to the first order integrability conditions in the
	procedure described above.
\end{remark}

The following lemma gives the first order integrability conditions in the
present situation.

\begin{lemma}
	\label{la:first.order.integrability}
	Assuming \eqref{eq:SIC:K:1st}, the Ricci identity for
	\eqref{eq:prolongation:T} reduces to
	\begin{multline}
		\label{eq:SIC:K:2nd}
		\frac{n+2}{9n}
		\left(
			\mathring T^{abc}\mathring T_{abc}
			-(n+2)(n-1)\bar t^a\bar t_a
			-9R
		\right)
		\bar t_i\\
		=
		-R_{,i}
		+\frac{2(5n+2)}{3(n-2)}\,\mathring T\indices{_i^{ab}}\mathring R_{ab}
		-\frac{2(n^2-2n+8)}{3(n-2)}\,\mathring R_{ia}\bar t^a.
	\end{multline}
\end{lemma}

\begin{proof}
	The Ricci identity for $T_{ijk}$ reads
	\[
	   	\young(l,m)
	   	T_{ijk,lm}
	   	=R\indices{^a_{ilm}}T_{ajk}
	   	+R\indices{^a_{jlm}}T_{iak}
		+R\indices{^a_{klm}}T_{ija}
	\]
	and can be reduced, by the help of \eqref{eq:prolongation:T}, to a cubic
	polynomial in $T_{ijk}$ involving the curvature tensor.  The trace
	free part of this cubic vanishes by virtue of \eqref{eq:SIC:K:1st}, as do
	the trace-free parts of its two independent traces.  The traces of these
	traces are both equivalent to \eqref{eq:SIC:K:2nd} using
	\eqref{eq:SIC:K:1st}.
\end{proof}

In Section~\ref{sec:R=const} we will show that for abundant superintegrable
systems on constant curvature manifolds the above procedure terminates after
the first step.  This generalises the findings in \cite{KKM07b} and
\cite{CK14,Capel,KKM07c}, where sufficiency is checked for dimensions~2 and~3,
respectively.

\section{The variety of superintegrable systems}
\label{sec:variety}

By definition, a second-order maximally superintegrable system consists of a
$(2n-1)$-dimensional subspace in the space $\mathcal K(M)$ of second-order
Killing tensors on the base manifold $M$ together with a potential function
$V$ on $M$.  Forgetting the latter defines a canonical map
\begin{equation}
	\label{eq:projection}
	\Phi:\mathcal S\to G_{2n-1}\bigl(\mathcal K(M)\bigr)
\end{equation}
from the set $\mathcal S$ of non-degenerate second-order maximally
superintegrable systems on $M$ to the corresponding Grassmannian.

\begin{proposition}
	\label{prop:equations}
	Let $\mathcal I\subset\mathcal S$ be the subset of non-degenerate
	irreducible superintegrable systems.  Then $\Phi(\mathcal I)$ is the
	subset in $G_{2n-1}\bigl(\mathcal K(M)\bigr)$ consisting of spaces of
	Killing tensors satisfying the following equations:
	\begin{enumerate}
		\item
			\label{it:rank}
			the maximal rank condition
			\[
				G_r(A)\not=0
				\qquad\text{for}\quad
				r=r_\text{max}=\frac{(n-1)(n+2)}2
			\]
			on an open and dense subset.
		\item
			\label{it:V}
			the generic integrability conditions \eqref{eq:SIC:V} of the
			prolongation equation \eqref{eq:prolongation:V} for a
			superintegrable potential
		\item
			\label{it:K}
			the integrability conditions \eqref{eq:SIC:K} of the prolongation
			equation \eqref{eq:prolongation:K} for a Killing tensor in a
			superintegrable system
	\end{enumerate}
\end{proposition}

\begin{proof}
	In previous sections we have already seen that the elements in
	$\Phi(\mathcal I)$ satisfy all three conditions: \ref{it:rank} is a
	consequence of Lemma~\ref{lemma:irreducible} together with \eqref{eq:rank}
	and \eqref{eq:rank:max}, \ref{it:V} is shown in
	Proposition~\ref{prop:SIC:V} and \ref{it:K} in
	Proposition~\ref{prop:SIC:K}.

	Now suppose a subspace of Killing tensors satisfies all three conditions.
	Condition~\ref{it:rank} assures that the tensor $T_{ijk}$, given by the
	Moore-Penrose inverse \eqref{eq:Moore-Penrose}, is well-defined on an open
	and dense subset of $M$.  Condition~\ref{it:V} implies that we can
	integrate Equations~\eqref{eq:prolongation:V} to obtain a linear space
	of solutions $V$ of dimension $n+2$.  Condition~\ref{it:K} assures that
	the Killing tensors all satisfy Equation~\eqref{eq:prolongation:K}.
	Together, the equations~\eqref{eq:prolongation:V} and
	\eqref{eq:prolongation:K} imply the Bertrand-Darboux
	condition~\eqref{eq:dKdV}, assuring the existence of potential functions
	$V^{(\alpha)}$ for each Killing tensor $K^{(\alpha)}$ with $V^{(0)}=V$ and
	$K^{(0)}=g$.

	Note that the corresponding integrals
	$F^{(\alpha)}=K^{(\alpha)}+V^{(\alpha)}$ define a superintegrable system
	only if they are functionally independent.  The following lemma shows that
	generically this is the case.  Therefore, for every space of Killing
	tensors satisfying all three conditions we actually find a superintegrable
	system which, by construction, is irreducible and non-degenerate and
	projects to this space under~$\Phi$.
\end{proof}

\begin{lemma}
	\label{lemma:functional_dependence}
	Let $K^{(\alpha)}$ be $2n-1$ linearly independent Killing tensors
	(including the metric) satisfying the integrability conditions
	\eqref{eq:SIC:K:generic} for \eqref{eq:prolongation:K}, and
	\eqref{eq:SIC:V} for~\eqref{eq:prolongation:V}.  Then, in the linear space
	of solutions $V$ to Equation~\eqref{eq:prolongation:V}, those $V$ defining
	functionally dependent integrals are confined to an affine subspace with
	non-empty complement.
\end{lemma}

\begin{proof}
	Suppose the first integrals \eqref{eq:quadratic} are functionally
	dependent.  This means there exists a function $\varphi:\mathbb
	R^{2n-1}\to\mathbb R$, non-zero on an open subset, such that
	\[
		\varphi(F^{(0)},\ldots,F^{(2n-2)})=0.
	\]
	Infinitesimally, this condition reads
	\[
		\sum_{\alpha=0}^{2(n-1)}
		\lambda_{(\alpha)} dF^{(\alpha)}
		=0,
	\]
	with
	\begin{align*}
		\lambda_{(\alpha)}(\mathbf p,\mathbf x)
		&=\frac{\partial\varphi}{\partial F^{(\alpha)}}
		(F^{(\alpha)}(\mathbf p,\mathbf x))&
		dF^{(\alpha)}
		&
		=\frac{\partial F^{(\alpha)}}{\partial x^k}dx^k
		+\frac{\partial F^{(\alpha)}}{\partial p^k}dp^k,
	\end{align*}
	where
	\begin{align*}
		\frac{\partial F^{(\alpha)}}{\partial x^k}&= K^{(\alpha)}_{ij,k}p^ip^j+V^{(\alpha)}_{,k}&
		\frac{\partial F^{(\alpha)}}{\partial p^k}&=2K^{(\alpha)}_{jk}p^j.
	\end{align*}
	Using \eqref{eq:potentials}, this can be written as
	\begin{subequations}\label{eq:independence}
		\begin{align}
			\label{eq:independence:dx}
			\sum_\alpha\lambda_{(\alpha)}\left(K^{(\alpha)}_{ij,k}p^ip^j+K^{(\alpha)}_{jk}V^{,j}\right)&=0\\
			\label{eq:independence:dp}
			\sum_\alpha\lambda_{(\alpha)} K^{(\alpha)}_{jk}p^j&=0.
		\end{align}
	\end{subequations}	
	Now, from the Killing equation \eqref{eq:Killing} and \eqref{eq:prolongation:K} we obtain
	\begin{align*}
		K^{(\alpha)}_{ij,k}p^ip^j
		&=-2K^{(\alpha)}_{kj,i}p^ip^j
		=-\frac23\young(jk,i)T\indices{^a_{jk}}K^{(\alpha)}_{ia}p^ip^j\\
		&=\frac23
			\left(
				T\indices{^a_{ij}}K^{(\alpha)}_{ka}-
				T\indices{^a_{kj}}K^{(\alpha)}_{ia}
			\right)
			p^ip^j.
	\end{align*}
	In the sum over $\alpha$, the second summand vanishes due to
	\eqref{eq:independence:dp},
	\[
		\sum_\alpha
			\lambda_{(\alpha)}
			K^{(\alpha)}_{ij,k}p^ip^j
		=\frac23
		\sum_\alpha
			\lambda_{(\alpha)}
			T\indices{^a_{ij}}K^{(\alpha)}_{ka}p^ip^j.
	\]
	Substituted back into \eqref{eq:independence:dx}, we have
	\[
		\sum_\alpha
			\lambda_{(\alpha)} K^{(\alpha)}_{ka}
			\left(
				\frac23T\indices{^a_{ij}}p^ip^j+
				V^{,a}
			\right)
		=0
	\]
	or
	\begin{equation}
		\label{eq:affine_subspace}
		V_{,a}
		\in
		-\frac23T\indices{_{aij}}p^ip^j+
		\ker\biggl(\sum_\alpha\lambda_{(\alpha)} K^{(\alpha)}\biggr).
	\end{equation}
	The solution space of \eqref{eq:prolongation:V} is parametrised by the
	values of $\nabla V$ and $\Delta V$ at a fixed point.  Therefore
	\eqref{eq:affine_subspace} proves the theorem in case the kernel is not
	maximal at some point in phase space.  On the other hand, if the kernel
	above is maximal for any point in phase-space, then the $K^{(\alpha)}$ are
	functionally linearly dependent and thus linearly dependent at some point
	in phase space.  By Proposition~\ref{prop:SIC:K} the $K^{(\alpha)}$ are
	determined by their values in a fixed point.  Hence they are linearly
	dependent in $\mathcal K(M)$, which contradicts the assumption.
\end{proof}

The proof of Proposition~\ref{prop:equations} shows that an explicit knowledge
of the set $\Phi(\mathcal I)$ solves the classification problem for
irreducible non-degenerate superintegrable systems.  This motivates the following
definition.

\begin{definition}
	We call the subset
	\[
		\Phi(\mathcal I)\subset G_{2n-1}\bigl(\mathcal K(M)\bigr)\,,
	\]
	given by the conditions in Proposition~\ref{prop:equations}, the
	\emph{classification space for irreducible non-degenerate superintegrable
	systems}.
\end{definition}

We are now ready to state our first main result.

\begin{theorem}
	\label{thm:main:variety}
	The classification space for irreducible non-degenerate superintegrable
	systems on a Riemannian manifold $M$ with analytic metric is a
	quasi-projective subvariety in the Grassmannian $G_{2n-1}\bigl(\mathcal
	K(M)\bigr)$ of $(2n-1)$-dimensional subspaces in the space $\mathcal K(M)$
	of Killing tensors on $M$.
\end{theorem}

\begin{proof}
	It suffices to show that in Proposition~\ref{prop:equations} the
	conditions \ref{it:V} and \ref{it:K} as well as the opposite of condition
	\ref{it:rank} define subvarieties in $G_{2n-1}\bigl(\mathcal K(M)\bigr)$.

	For the moment, let us fix a point~$x\in M$.  The evaluation~$K\mapsto
	K(x)$ of a tensor $K$ at $x$ and the covariant derivative~$K\mapsto\nabla
	K$ are linear operations.  Therefore, for every~$x$, the components of a
	Killing tensor and its derivatives,
	\begin{align*}
		K_{ij}&(x)&
		K_{ij,k}&(x)&
		K_{ij,kl}&(x),
	\end{align*}
	are linear functions on the space~$\mathcal K(M)$ of Killing tensors
	on~$M$.

	The Moore-Penrose inverse $A^\dagger$ is rational in the components of
	$A$, with homogeneous numerator and denominator of degree $2r-1$
	respectively $2r$.  Consequently, the components $T_{ijk}(x)$ of the
	structure tensor of an irreducible superintegrable system are rational
	functions in the $K^{(\alpha)}_{ij}(x)$ and the $K^{(\alpha)}_{ij,k}(x)$
	for $\alpha=0,1,\ldots,2n-2$.  Hence the $T_{ijk}(x)$ are rational
	functions on the space $\mathcal K(M)^{2n-1}$, homogeneous of degree $0$.
	Similarly, the components $T_{ijk,l}(x)$ of the derivative of the
	structure tensor are rational in $K^{(\alpha)}_{ij}(x)$,
	$K^{(\alpha)}_{ij,k}(x)$ and $K^{(\alpha)}_{ij,kl}(x)$, so they as well
	are homogeneous rational functions of degree 0 on $\mathcal K(M)^{2n-1}$.

	We now show that in Proposition~\ref{prop:equations} the conditions
	\ref{it:V} and \ref{it:K} as well as the opposite of condition
	\ref{it:rank} are given by homogeneous algebraic equations.  First note
	that a set of Killing tensors has non-maximal rank if and only if
	\[
		G_r(A)=0
		\qquad\text{for}\quad
		r=r_\text{max}=\frac{(n-1)(n+2)}2
	\]
	holds on all of $M$.  Moreover, $G_r(A)$ is a homogeneous polynomial in
	the components of $A$, i.e.\ in the $K^{(\alpha)}_{ij}(x)$.  This shows
	that the opposite of condition~\ref{it:rank} in
	Proposition~\ref{prop:equations} is a homogeneous algebraic equation on
	$\mathcal K(M)^{2n-1}$ for every $x\in M$.

	Evaluated at $x$, the integrability conditions~\eqref{eq:SIC:V} are
	polynomial in $T_{ijk}(x)$ and $T_{ijk,l}(x)$, which are homogeneous
	rational functions of degree 0 on $\mathcal K(M)^{2n-1}$.  Consequently,
	the integrability conditions~\eqref{eq:SIC:V} can be written as
	homogeneous algebraic equations on $\mathcal K(M)^{2n-1}$ for every $x\in
	M$.  By similar arguments the same is true for the integrability
	conditions~\eqref{eq:SIC:K}.

	To summarise, in Proposition~\ref{prop:equations} the conditions
	\ref{it:V} and \ref{it:K} as well as the opposite of condition
	\ref{it:rank} are given by homogeneous algebraic equations on the space
	$\mathcal K(M)^{2n-1}$ for every point $x\in M$.  Together, these
	equations define a quasi-projective subvariety in the Stiefel manifold
	$V_{2n-1}\bigl(\mathcal K(M)\bigr)$.

	The Bertrand-Darboux equations~\eqref{eq:dKdV} as well as the properties of irreducibility
	and non-degeneracy are invariant under linear changes of the basis
	$K^{(\alpha)}$.  Therefore, the above subvariety is invariant under the
	action of the general linear group $\operatorname{GL}_{2n-1}$ on
	$V_{2n-1}\bigl(\mathcal K(M)\bigr)$ and descends to a quasi-projective
	subvariety of the Grassmannian
	\[
		G_{2n-1}\bigl(\mathcal K(M)\bigr)
		\cong
		\frac{V_{2n-1}\bigl(\mathcal K(M)\bigr)}{\operatorname{GL}_{2n-1}}.
	\]
\end{proof}

\section{Constant curvature}
\label{sec:R=const}

\subsection{Codazzi tensors of a superintegrable system}

In this section we express the integrability conditions \eqref{eq:SIC:V:bis}
on constant curvature manifolds in terms of Codazzi tensors and use their
local form to rewrite them as a system of partial differential equations for
two scalar functions -- the \emph{structure functions} of the superintegrable
system.

\begin{definition}
	A \emph{second-order Codazzi tensor} is a symmetric tensor $C_{ij}$
	satisfying
	\begin{equation}
		\label{eq:Codazzi:2nd}
		\young(j,k)C_{ij,k}=0.
	\end{equation}
\end{definition}

\begin{proposition}
	\label{prop:Codazzi:two}
	For every non-degenerate superintegrable system on a constant curvature
	manifold of dimension $n\geq3$, there exists a function $\zeta$ such that the trace modification
	\begin{equation}
		\label{eq:Cij}
		C_{ij}=Z_{ij}+\zeta g_{ij}
	\end{equation}
	of the symmetric tensor~\eqref{eq:Z} is a Codazzi tensor.
\end{proposition}

\begin{proof}
	On a constant curvature manifold we have
	\[
		R_{ijkl}
		=
		\frac R{n(n-1)}
		(g_{ik}g_{jl}-g_{il}g_{jk}),
	\]
	where $R=R\indices{^{ab}_{ab}}$ is the scalar curvature.  In this case the
	curvature term in \eqref{eq:SIC:integrability} vanishes due to the first
	integrability condition~\eqref{eq:SIC:V:linear},
	\[
		\young(ij)\young(k,l,m)
		Z_{ik,l}g_{jm}
		=0.
	\]
	The trace of this equation over $(j,k)$ yields a conformal Codazzi
	equation for the tensor $Z_{ij}$:
	\begin{equation}
		\label{eq:Codazzi:conformal}
		{\young(li,m)}^*_{\mspace{-26mu}\circ\mspace{+26mu}}
		Z_{il,m}
		=0.
	\end{equation}
	Skew symmetrising the covariant derivative of this equation in all indices
	but $i$ and using the Ricci identity results in
	\[
		{\young(li,m,n)}^*
		\left(
			R\indices{^a_{lmn}}Z_{ai}
			+\frac1{n-1}g_{il}Z\indices{_m^a_{,an}}
		\right)
		=0.
	\]
	Taking the trace of this equation over $(i,n)$ now shows that the divergence
	of $Z_{ij}$ is closed,
	\[
		\young(l,m)Z\indices{_m^a_{,al}}=0,
	\]
	and hence the differential of some function, more precisely
	\[
		Z\indices{_m^a_{,a}}=(n-1)\zeta_{,m}.
	\]
	Under this condition, the Codazzi equation~\eqref{eq:Codazzi:2nd}
	for $Z_{ij}$ reduces to the conformal Codazzi equation
	\eqref{eq:Codazzi:conformal} for $C_{ij}$, which we have already shown to
	be satisfied.
\end{proof}

\begin{lemma}\cite{Ferus}
	\label{lemma:Codazzi:two}
	On a constant curvature manifold every second-order Codazzi tensor
	$C_{ij}$ is locally of the form
	\[
		C_{ij}=C_{,ij}+\frac R{n(n-1)}Cg_{ij},
	\]
	for some function $C$, where $R$ is the scalar curvature.  Conversely,
	every tensor of this form on a constant curvature manifold is a Codazzi
	tensor.
\end{lemma}

\begin{definition}
	A \emph{third order Codazzi tensor} is a totally symmetric tensor
	$B_{ijk}$ that satisfies
	\begin{equation}
		\label{eq:Codazzi:3rd}
		\young(k,l)B_{ijk,l}=0.
	\end{equation}
\end{definition}

\begin{proposition}
	\label{prop:Codazzi:three}
	For every non-degenerate superintegrable system on a constant curvature
	manifold of dimension $n\geq3$, the tensor
	\begin{equation}
		\label{eq:Bijk}
		 B_{ijk}
		=T_{ijk}
		+\frac1{n-1}g_{ij}t_{,k}
		+\frac1{2(n-2)}\young(ijk)g_{ij}C_{,k}
	\end{equation}
	is a Codazzi tensor.
\end{proposition}

\begin{proof}
	First note that the tensor $B_{ijk}$ is indeed symmetric, due to the first
	integrability condition~\eqref{eq:SIC:V:linear}.  Using the definitions
	\eqref{eq:Cij} and \eqref{eq:Bijk} of $C_{ij}$ respectively $B_{ijk}$
	together with the fact that
	\[
		\young(ij)\young(k,l)g_{ik}g_{jl}=0,
	\]
	we check that the Codazzi equation for $B_{ijk}$ is equivalent to the
	second integrability condition~\eqref{eq:SIC:V:quadratic}:
	\[
	 	\young(k,l)B_{ijk,l}
		=
		\frac12\young(ij)\young(k,l)
		\left(
			T_{ijk,l}
	 		+\frac2{n-2}g_{ik}Z_{jl}
		\right)
		=0.
	\]
\end{proof}

The proof of the following Lemma is analogous to that of
Lemma~\ref{lemma:Codazzi:two}.

\begin{lemma}
	\label{lemma:Codazzi:three}
	On a constant curvature manifold every third order Codazzi tensor
	$B_{ijk}$ is locally of the form
	\[
		B_{ijk}=\frac16\young(ijk)\left(B_{,ij}+\frac{4R}{n(n-1)}g_{ij}B\right)_{,k}
	\]
	for some function $B$.  Conversely, every tensor of this form on a
	constant curvature manifold is a Codazzi tensor.
\end{lemma}

We can now rewrite the structure tensor and all integrability conditions in
terms of the structure functions $B$ and $C$.

\begin{proposition}
	The structure tensor of a superintegrable system on a constant curvature
	manifold of dimension $n\geq3$ has the decomposition~\eqref{eq:St2T} with
	\begin{subequations}
		\label{eq:BC2T}
		\begin{align}
			\label{eq:B2S}
			\mathring T_{ijk}&=\frac16{\young(ijk)\!}_\circ\,B_{,ijk}\\
			\label{eq:BC2t}
			\frac{n}{n-1}t&=\Delta B+\frac{2(n+1)}{n(n-1)}RB-\frac{n+2}{n-2}C+\text{constant}
		\end{align}
	\end{subequations}
	for two functions $B$ and $C$, which are unique up to a gauge
	transformation
	\begin{subequations}
		\label{eq:gauge}
		\begin{align}
			\label{eq:gauge:B}B&\mapsto B+\delta B&{\young(ijk)\!}_\circ\,&\delta B_{,ij         k}=0\\
			\label{eq:gauge:C}C&\mapsto C+\delta C&{\young(ij )\!}_\circ\,&\delta C_{,ij\phantom k}=0
		\end{align}
		satisfying the compatibility condition
		\begin{equation}
			\label{eq:gauge:BC}
			\frac{n+2}{n-2}\delta C
			=\Delta\delta B
			+\frac{2(n+1)}{n(n-1)}R\delta B
			+\text{constant}.
		\end{equation}
	\end{subequations}
	In particular, we can choose simultaneously
	\begin{subequations}
		\label{eq:gauge:0}
		\begin{align}
			C(x_0)&=0&
			\nabla C(x_0)&=0&
			\Delta C(x_0)&=0\\
			B(x_0)&=0&
			\nabla       B(x_0)&=0&
			\nabla\nabla B(x_0)&=0
		\end{align}
	\end{subequations}
	in a fixed point $x_0$.
\end{proposition}

\begin{proof}
	Combining Proposition~\ref{prop:Codazzi:three} and
	Lemma~\ref{lemma:Codazzi:three}, we get
	\[
		T_{ijk}+\frac1{n-1}g_{ij}t_{,k}
		=
		\young(ijk)
		\left(
			\frac16B_{,ijk}
			+\frac{2R}{3n(n-1)}g_{ij}B_{,k}
			-\frac1{2(n-2)}g_{ij}C_{,k}
		\right).
	\]
	Taking the trace-free part on each side results in~\eqref{eq:B2S}.
	Contracting this equation in $i$ and $j$ yields
	\[
		\frac n{n-1}t_{,k}
		=\frac13
			\left(
				  B\indices{_{,a}^a_k}
				+2B\indices{_{,ka}^a}
			\right)
		+\frac{4R}{3n(n-1)}B_{,k}
		-\frac{n+2}{n-2}C_{,k}.
	\]
	By the Ricci identity,
	\[
		 B\indices{_{,ka}^a}
		-B\indices{_{,a}^a_k}
		=
		 B\indices{_{,a}_k^a}
		-B\indices{_{,a}^a_k}
		=
		R\indices{^b_{ak}^a}B_{,b}
		=
		\frac{R}{n}B_{,k},
	\]
	so that
	\[
		\frac n{n-1}t_{,k}
		=B\indices{_{,a}^a_k}
		+\frac{2(n+1)R}{n(n-1)}B_{,k}
		-\frac{n+2}{n-2}C_{,k}.
	\]
	This is equivalent to \eqref{eq:BC2t} and imposes the
	constraint~\eqref{eq:gauge:BC} on the gauge transforms \eqref{eq:gauge:B}
	and \eqref{eq:gauge:C}.

	A flat constant curvature manifold is locally isometric to Euclidean
	space, so that \eqref{eq:gauge:B} and \eqref{eq:gauge:C} can easily be
	integrated to give
	\begin{align*}
		\delta C(\vec r)&=2(n-2)\left(\tfrac12c_2r^2+\vec c_1\vec r+c_0\right)\\
		\delta B(\vec r)&=\tfrac14b_4r^4+r^2\vec b_3\vec r+\vec r\:^T\!\!A\vec r+\vec b_1\vec r+b_0,
	\end{align*}
	where the $b_0,b_4,c_0,c_2$ are scalar constants, $\vec b_1,\vec b_3,\vec
	c_1$ are vectorial constants and $A$ is a constant symmetric matrix.
	Imposing~\eqref{eq:gauge:BC} with $R=0$, we find $b_4=c_2$ and $\vec
	b_3=\vec c_1$.  This shows that we can chose the gauge \eqref{eq:gauge:0}
	in a single point.

	A non-flat constant curvature manifold is conformally equivalent to a flat
	one.  Noting that on a constant curvature manifold the operators
	\eqref{eq:gauge:BC} are invariant under conformal transformations and that
	the Laplace operator $\Delta$ changes by a multiple of the identity, the
	proof is similar to the flat case.
\end{proof}

\begin{proposition}
	On a manifold of constant curvature $\kappa$ and dimension $n\geq3$,
	the integrability conditions
	for a superintegrable potential in the form \eqref{eq:SIC:V:bis} are
	equivalent to the following equations for the structure functions $B$
	and~$C$:
	\begin{subequations}
		\begin{align}
			\label{eqn:Weyl.BC}
			0&=
			{\young(ik,jl)}^*_\circ B\indices{_{,ik}^a}B_{,jla}\\
			\mathring C_{,ij}
			&\notag= \bigg[
				B\indices{_{,i}^{ab}}B_{,jab}
				+B\indices{_{,ij}^a} \Bigl( C - 2(n-2)\kappa B - \Delta B \Bigr)_{,a}
			\\ &\qquad
				-\frac1{n-2}\,\Bigl( C-(n-2)\kappa B \Bigr)_{,i} \Bigl( C-(n-2)\kappa B \Bigr)_{,j}
			\bigg]_\circ
			\label{eqn:consistency.BC}
		\end{align}
	\end{subequations}
\end{proposition}

\begin{proof}
	This follows from substituting \eqref{eq:St2T} with \eqref{eq:BC2T}
	into \eqref{eq:SIC:V:bis}:  Equation~\eqref{eq:SIC:V:symmetries} is
	identically satisfied by Proposition~\ref{prop:SIC:V:linear}, Equation
	\eqref{eq:SIC:V:Weyl} is equivalent to Equation~\eqref{eqn:Weyl.BC},
	Equation~\eqref{eq:SIC:V:differential} is satisfied due to the Codazzi
	equation for $B_{ijk}$ and Equation~\eqref{eq:Z} is equivalent to
	\eqref{eqn:consistency.BC}, using \eqref{eq:Cij}.
\end{proof}

\begin{remark}
	Equation~\eqref{eqn:Weyl.BC} can alternatively be written with the
	(symmetric) Codazzi tensor $B_{ijk}$ instead of the covariant derivative
	$B_{,ijk}$ of the structure function~$B$:
	\[
		0={\young(ik,jl)}^*_\circ B\indices{_{ik}^a}B_{jla}.
	\]
	The reason is that $B_{,ijk}$ is symmetric up to trace terms, so that the
	only non-trace terms in the contraction $B\indices{_{ik}^a}B_{jla}$ do not
	have Riemann symmetry according to the Littlewood-Richardson rule
	\[
		\yng(3)\otimes\yng(1)\cong\yng(4)\oplus\yng(3,1).
	\]
\end{remark}

\begin{remark}
	Equation~\eqref{eqn:consistency.BC} can be extended by a second equation
	expressing the covariant derivative of the Laplacian $\Delta C$
	polynomially in $\nabla C$, where the coefficients are polynomial in the
	derivatives of $B$.  Together, they define a (non-linear) prolongation of
	\eqref{eqn:consistency.BC}.  This leads to higher order integrability
	conditions for the prolongation of a superintegrable Killing tensor on a
	constant curvature manifold and will be subject of a forthcoming paper.
\end{remark}

\subsection{Algebraic superintegrability conditions for abundant systems}
\label{sec:abundant}

\begin{proposition}
	\label{prop:necessary}
	The structure tensor of an abundant superintegrable system on a constant
	curvature manifold of dimension $n\geq3$ satisfies
	\begin{subequations}
		\label{eq:SIC:K:abundant}
		\begin{align}
			\label{eq:SIC:K:Weyl:abundant}
			{\young(ik,jl)}^*_\circ\mathring T\indices{^a_{ik}}\mathring T_{ajl}&=0\\
			\label{eq:SIC:K:Ricci:abundant}
			{\young(ij)}_\circ
			\left(
			\mathring T\indices{_i^{ab}}\mathring T_{jab}
			-(n-2)(\mathring T\indices{_{ij}^a}\bar t_a+\bar t_i\bar t_j)
			\right)
			&=0\\
			\label{eq:perfect_square}
			\mathring T^{abc}\mathring T_{abc}-(n-1)(n+2)\bar t^a\bar t_a&=9R
		\end{align}
	\end{subequations}
\end{proposition}

\begin{proof}
	For constant curvature, Equations~\eqref{eq:SIC:K:1st} imply
	and \eqref{eq:SIC:K:Weyl:abundant} and \eqref{eq:SIC:K:Ricci:abundant},
	while Equation \eqref{eq:SIC:K:2nd} reads
	\[
		\left(
			\mathring T^{abc}\mathring T_{abc}
			-(n+2)(n-1)\bar t^a\bar t_a
			-9R
		\right)
		\bar t_i
		=0,
	\]
	showing that \eqref{eq:perfect_square} holds over the support of $\bar
	t_i$.  So let us suppose $\bar t_i=0$ in a local neighbourhood of some
	point.  There the tracefree and trace part of \eqref{eq:DT:Dt} read
	\begin{align*}
		\left(\mathring T\indices{_i^{ab}}\mathring T_{jab}\right)_\circ&=0&
		\mathring T^{abc}\mathring T_{abc}&=-\frac{9(n+2)}{3n+2}R
	\end{align*}
	and show that \eqref{eq:perfect_square} holds for $R=0$.  We therefore
	suppose $R\not=0$ from now on.  With the above equations \eqref{eq:DT:DS}
	reads
	\[
		\mathring T_{ijk,l}
		=\frac1{18}{\young(ijk)}
		\left(
			\mathring T\indices{_{ij}^a}\mathring T_{kla}
			+\frac{18R}{(3n+2)n}g_{ik}g_{jl}
		\right).
	\]
	We can use this equation to eliminate all derivatives in the Ricci
	identity
	\[
		\young(m,l)\mathring T_{ijk,lm}
		=\young(ijk)R\indices{^a_{ilm}}\mathring T_{ajk}
	\]
	to give
	\[
		\frac R{n(n-1)}\young(ijk)\young(m,l)\mathring T_{ijm}g_{kl}=0.
	\]
	A contraction over $(k,l)$ then shows
	\[
		\frac R{n-1}\mathring T_{ijm}=0.
	\]
	For $R\not=0$ this implies $\mathring T_{ijk}=0$.  This means that the
	structure tensor vanishes, which is only possible for $R=0$.
\end{proof}

The following proposition shows that the necessary
conditions~\eqref{eq:SIC:K:abundant} are sufficient to reconstruct an abundant
superintegrable system from the values of its structure tensor in a single
point.

\begin{proposition}
	\label{prop:sufficient}
	Consider a manifold $M$ of constant curvature and dimension $n\geq3$.
	Assume we are provided with the values of a tensor $T_{ijk}$
	in a fixed point $x_0\in M$ such that
	$T_{ijk}$ satisfies Conditions~\eqref{eq:SIC:K:abundant} in $x_0$.
	Then there exists a solution $T_{ijk}$ to the non-linear prolongation
	equations~\eqref{eq:prolongation:T}, defined almost everywhere in a
	neighborhood of $x_0$, where $T_{ijk}$ also satisfies the algebraic
	equations~\eqref{eq:SIC:K:abundant}.
\end{proposition}

\begin{proof}
	Using~\eqref{eq:prolongation:T}, the covariant derivatives of the three
	algebraic conditions~\eqref{eq:SIC:K:abundant} can be written as
	polynomials in $T_{ijk}$.  It is straightforward to check that these
	polynomials lie in the ideal generated algebraically
	by~\eqref{eq:SIC:K:abundant}.
\end{proof}

Let us briefly retrace how an abundant superintegrable system can be
reconstructed from the values of its structure tensor in a single point.

\begin{enumerate}
	\item
		Given a solution $T_{ijk}(x_0)$ to the algebraic integrability
		conditions~\eqref{eq:SIC:K:abundant} in a single point $x_0$, one can
		solve the prolongation equation~\eqref{eq:prolongation:T} to extend
		this solution to a tensor $T_{ijk}(x)$ satisfying the
		conditions~\eqref{eq:SIC:K:abundant} in a neighbourhood of $x_0$
		(c.f.~Proposition~\ref{prop:sufficient}).
	\item
		Next, one can solve the prolongation
		equation~\eqref{eq:prolongation:K} for any initial values
		$K_{ij}(x_0)$, yielding an $n(n+1)/2$-dimensional space of Killing
		tensors~$K_{ij}(x)$, as the corresponding integrability conditions are
		satisfied by the prolongation equation~\eqref{eq:prolongation:T} and
		a subset of the algebraic integrability
		conditions~\eqref{eq:SIC:K:abundant}
		(c.f.~Proposition~\ref{prop:prolongation:T}).
	\item
		Then, one can solve the prolongation
		equation~\eqref{eq:prolongation:V} for any initial values $V(x_0)$,
		$\nabla V(x_0)$ and $\Delta V(x_0)$ to obtain a function $V(x)$, since
		the integrability conditions for a superintegrable Killing tensor
		imply those for a superintegrable potential
		(c.f.~Corollary~\ref{cor:SIC:K->V}).
	\item
		As the prolongation equation~\eqref{eq:prolongation:K} for a
		superintegrable Killing tensor is the obstruction that a solution to
		the prolongation equation~\eqref{eq:prolongation:V} for a
		superintegrable potential solves the Bertrand-Darboux
		equation~\eqref{eq:dKdV}, the latter is satisfied as well.  This in
		turn is the integrability condition for Equation~\eqref{eq:potentials}.
		Therefore, one can solve the latter in order to obtain the individual
		potentials $V^{(\alpha)}$ that complement the Killing tensors
		$K^{(\alpha)}$ to integrals $F^{(\alpha)}=K^{(\alpha)}+V^{(\alpha)}$
		of the Hamiltonian $H=F^{(0)}=g+V$.  For a generic solution $V$ and a
		generic choice of $(2n-1)$ linearly independent Killing tensors, these
		integrals are functionally independent and hence form a
		superintegrable system (c.f.~Lemma~\ref{lemma:functional_dependence}).
\end{enumerate}
By construction, the superintegrable system thus obtained is non-degenerate,
and in fact abundant. On simply connected constant curvature manifolds the
locally defined Killing tensors extend globally, so the above construction is
global in this case.

\begin{corollary}
	A superintegrable system on a constant curvature manifold of dimension
	$n\geq3$ is abundant if and only if it satisfies the algebraic
	equations~\eqref{eq:SIC:K:abundant}.
\end{corollary}

\begin{proof}
	By Proposition~\ref{prop:necessary}, an abundant superintegrable system
	satisfies Equations~\eqref{eq:SIC:K:abundant}.  Conversely, suppose a
	superintegrable system satisfies these equations.  Then by
	Proposition~\ref{prop:sufficient}, there is an abundant superintegrable
	system having the same values of the structure tensor and of the Killing
	tensors at the fixed point~$x_0$.  Due to
	Proposition~\ref{prop:uniqueness} both systems the Killing tensors and the
	structure tensors of both systems coincide everywhere.  In particular,
	both are abundant.
\end{proof}

Theorem~\ref{thm:main:variety} states that the configuration space of
non-degenerate superintegrable systems is a quasi-projective variety.  Our
proof is not constructive, as it is based on an \emph{infinite} set of
\emph{implicitly defined} algebraic equations.  For an abundant system on a
constant curvature manifold, Proposition~\ref{prop:sufficient} reduces this
set to a \emph{finite} set of \emph{explicit} algebraic equations, given
by~\eqref{eq:SIC:K:abundant}.  This is what renders an algebraic-geometric
classification tractable in the first place.  We can simplify these equations
further by rewriting them in terms of the structure functions.

\begin{corollary}
	\label{cor:Master:B}
	A superintegrable system on a constant curvature manifold of dimension
	$n\geq3$ is abundant if and only if the structure function $B$ satisfies
	\begin{equation}
		\label{eq:Master:B}
		{\young(ij,kl)}^*
		\left(
			B\indices{^a_{ij}}B_{akl}
			+\frac{9R}{n(n-1)}g_{ij}g_{kl}
		\right)=0
	\end{equation}
	and the structure function $C$ vanishes, up to gauge
	transforms of the form~\eqref{eq:gauge}.
\end{corollary}

\begin{proof}
	Observe that for constant curvature the left hand side of
	Equation~\eqref{eq:SIC:K:Ricci:abundant} is nothing but the trace free
	part of the tensor~$Z_{ij}$ defined in~\eqref{eq:Z} and used in the
	definition~\eqref{eq:Codazzi:2nd} of the Codazzi tensor $C_{ij}$.  This
	shows that the trace-free part of $C_{ij}$ is zero and that we can choose
	a gauge in which $C$ vanishes.  But for $C$ identically zero,
	Equation~\eqref{eq:SIC:K:abundant} becomes the Ricci decomposition of the
	algebraic curvature tensor in~\eqref{eq:Master:B} after expressing the
	structure tensor in terms of the structure functions via~\eqref{eq:BC2T}.
\end{proof}

\subsection{Structure connection of a superintegrable system}

We now encode the information about a superintegrable system on a constant
curvature manifold in a torsion-free affine connection and express all
relevant integrability conditions for abundant systems as the flatness of this
connection.

\begin{definition}
	For a non-degenerate superintegrable system on a manifold of constant
	curvature and dimension $n\geq3$, we introduce a torsion-free affine
	connection~$\hat\nabla$ by
	\[
		g(\hat\nabla_X Y,Z) = g(\nabla_X Y,Z) + A(X,Y,Z),
	\]
	where the tensor $A_{ijk}$ is given by
	\begin{equation}
		\label{eqn:Aijk}
		A_{ijk} := -\frac13\,\left( T_{ijk} + \frac{1}{n-1}\,g_{ij}t_k \right)\,.
	\end{equation}
	We call $\hat\nabla$ the
	\emph{structure connection} of the superintegrable system and denote its
	curvature tensor by $\hat R_{ijkl}$.
\end{definition}

Note that $A_{ijk}$ is symmetric due to the first integrability
condition~\eqref{eq:SIC:V:linear}.

\begin{proposition}
	The curvature of the structure connection $\hat\nabla$, written in terms
	of the structure functions, reads
	\begin{multline}
		\label{eq:BC2R}
		\hat R\indices{^i_{jkl}}
		= R\indices{^i_{jkl}}
		+ \frac19\,{\young(k,l)}\,\bigg(
			B\indices{^i_{ka}}B\indices{^a_{jl}}
		\\
			+\frac{C^{,a}}{n-2} \left( g\indices{^i_l} B_{jak} - g_{jl} B\indices{^i_{ak}} \right)
			+\frac{3}{n-2}\,\left( g\indices{^i_l}C_{,jk} + g_{jl} C\indices{^{,i}_k} \right)
		\\
			+\frac1{(n-2)^2} \left(
					  g\indices{^i_k} C_{,j} C_{,l}
					  - g_{jk} C^{,i} C_{,l}
					  + g\indices{^i_k} g_{jl} C^{,a} C_{,a}
					\right)
			\bigg),
	\end{multline}
	where $R\indices{^i_{jkl}}$ is the metric curvature tensor.
\end{proposition}

\begin{proof}
	The formula follows from applying the definition of the curvature of an affine
	connection,
	\[
		\hat R\indices{^i_{jkl}}
		=R\indices{^i_{jkl}}
		+\young(k,l)
		\left(
			\nabla_kA\indices{^i_{jl}}
			+A\indices{^i_{mk}}A\indices{^m_{jl}}
		\right),
	\]
	taking into account the Codazzi equation~\eqref{eq:Codazzi:3rd} for $B_{ijk}$.
\end{proof}

\begin{corollary}
	A superintegrable system on a constant curvature manifold of dimension
	$n\geq3$ is abundant if and only if its stucture connection is flat.
\end{corollary}

\begin{proof}
	We can decompose the curvature tensor $\hat R_{ijkl}$ into its symmetric
	and antisymmetric part with respect to the indices $(i,j)$.  Due to the
	symmetry of $A_{ijk}$ the antisymmetric part has Riemann symmetry.  The
	symmetric part is proportional to
	\[
		{\young(kji,l)}^*g_{ij}C_{,kl}.
	\]
	Taking the trace in $(i,k)$ shows that this term vanishes if and only if
	$C_{,kl}$ is trace.  The latter implies that we can choose a gauge in
	wich $C$ vanishes.  But for $C$ identically zero, $\hat R_{ijkl}$ becomes
	the curvature tensor in \eqref{eq:Master:B}.
\end{proof}

\subsection{The variety of abundant superintegrable systems}

In the previous sections we have shown how an abundant superintegrable system
can be reconstructed from the values of its structure tensor in a single point
and we have given explicit algebraic equations which are necessary and
sufficient conditions for a tensor to be the structure tensor of an abundant
superintegrable system.  In order to understand the corresponding variety and
its relation to the quasi-projective variety defined in
Section~\ref{sec:variety}, we first need to comment on a technical subtlety.

Up to now we required a superintegrable system to be irreducible.  This assumption enabled us to define the structure tensor in terms of the Killing
tensors in the superintegable system, given by the Moore-Penrose inverse.
Every time we have been writing down the structure tensor, we were considering it to be an
explicit function of the Killing tensors in the superintegrable system.  For an abundant system, in contrast, we can now give an alternative, implicit way of defining the
structure tensor, namely by using the non-linear
prolongation~\eqref{eq:prolongation:T} and requiring the corresponding
algebraic integrability conditions~\eqref{eq:SIC:K:abundant} to hold.  Note
that superintegrable systems arising in this way are always non-degenerate,
but need not be irreducible (see Figure~\ref{fig:Venn}).  Nevertheless,
if they are irreducible, then obviously both definitions coincide.

\begin{figure}[ht]
	\centering
	\includegraphics[width=.375\textwidth]{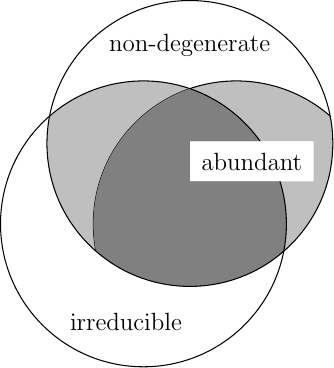}
	\caption{}
	\label{fig:Venn}
\end{figure}

With these clarifications made, we can illustrate the relation between the two varieties
geometrically as sketched in Figure~\ref{fig:varieties}:  Denote by
$F_{n_1,n_2}\bigl(\mathcal K(M)\bigr)$ the flag variety of inclusions
$U_1\subset U_2\subset\mathcal K(M)$ of subspaces $U_1$ and $U_2$ in the space
of Killing tensors with dimensions $\dim U_1=n_1=2n-1$ and $\dim
U_2=n_2=\frac12n(n+1)$.  Forgetting either one of these two subspaces,
we have the two projections
\begin{align*}
	\pi_i&:F_{n_1,n_2}\bigl(\mathcal K(M)\bigr)\longrightarrow G_{n_i}\bigl(\mathcal K(M)\bigr)&
	i&=1,2
\end{align*}
onto the corresponding Grassmannians.
Denote by $\mathcal I\subseteq\mathcal S$ the set of
\emph{irreducible} non-degenerate superintegrable systems.
In Section~\ref{sec:variety} we have shown that its image under the canonical
map
\[
	\Phi_1:\mathcal I\longrightarrow G_{n_1}\bigl(\mathcal K(M)\bigr)
\]
is a quasi-projective variety.  In the same way, we can define a canonical map
\[
	\Phi_2:\mathcal A\longrightarrow G_{n_2}\bigl(\mathcal K(M)\bigr)
\]
from the set $\mathcal A\subseteq\mathcal S$ of \emph{abundant} non-degenerate
superintegrable systems to the Grassmannian $G_{n_2}\bigl(\mathcal
K(M)\bigr)$, by mapping such a system to the space spanned by all its
\emph{linearly} independent Killing tensors.

\begin{figure}[ht]
	\[
		\xymatrix@C=0pt@M=6pt{
			n_1=2n-1&&
			F_{n_1,n_2}\bigl(\mathcal K(M)\bigr)\ar[dl]_{\pi_1}\ar[dr]^{\pi_2}&&
			n_2=\frac{n(n+1)}2\\
			\mathcal I\ar[r]^(.33){\Phi_1}&
			G_{n_1}\bigl(\mathcal K(M)\bigr)&&
			G_{n_2}\bigl(\mathcal K(M)\bigr)&
			\mathcal A\ar[l]_(.33){\Phi_2}
		}
	\]
	\caption{}
	\label{fig:varieties}
\end{figure}

\noindent As we are going to see shortly, the structure tensor~$T_{ijk}$ of an abundant system
can be obtained from its image under $\Phi_2$ by solving the prolongation
equation \eqref{eq:prolongation:K} for $T_{ijk}$.  The potential can then be
recovered by integrating the prolongation \eqref{eq:prolongation:V}.  That is,
an explicit knowledge of the set $\Phi_2(\mathcal A)$ solves the
classification problem for abundant non-degenerate superintegrable systems.
This motivates the following definition.

\begin{definition}
	We call the subset
	\[
		\Phi_2(\mathcal A)\subset G_{\frac{n(n+1)}2}\bigl(\mathcal K(M)\bigr)
	\]
	the \emph{classification space for abundant non-degenerate superintegrable
	systems}.
\end{definition}

Note that for abundant systems we have a fibre bundle
\[
	\Phi_1(\mathcal A)\longrightarrow
	\Phi_2(\mathcal A),
\]
whose fibres consist of superintegrable systems possessing the same structure
tensor and are isomorphic to the Grassmannian $G_{n_1}(n_2)$.

We are now ready to state our second main result.

\begin{theorem}
	The classification space for abundant non-degenerate superintegrable
	systems on a Riemannian manifold $M$ is a subvariety in the Grassmannian
	$G_{\frac{n(n+1)}2}\bigl(\mathcal K(M)\bigr)$ of $n(n+1)/2$-dimensional
	subspaces in the space $\mathcal K(M)$ of Killing tensors on $M$.  It is
	isomorphic to the variety of cubic forms
	\[
		\Psi_{ijk}x^ix^jx^k
	\]
	on $\mathbb R^n$ satisfying
	\begin{equation}
		\label{eq:Master:Psi}
		\young(j,k)
		\left(
			\Psi\indices{^a_{ij}}\Psi_{akl}
			+\frac R{n(n-1)}g_{ij}g_{kl}
		\right)=0.
	\end{equation}
\end{theorem}

\begin{proof}
	For a fixed point $x_0\in M$, let $Y$ be the variety of cubic forms on
	$T_{x_0}M\cong\mathbb R^n$ satisfying~\eqref{eq:Master:Psi}.

	We construct the isomorphism $f:Y\to\Phi_2(\mathcal A)$ from a map
	\[
		\Sym^3T_{x_0}M\times\Sym^2T_{x_0}M\to\mathcal K(M),
	\]
	sending a pair $(\Psi,S)$ to the Killing tensor $K\in\mathcal K(M)$ with
	\begin{itemize}
		\item $K_{ij}(x_0)=S_{ij}$,
		\item $K_{ij,k}(x_0)$ given by \eqref{eq:prolongation:K} and
		\item $K_{ij,kl}(x_0)$ given by \eqref{eq:K''},\footnote{%
				Recall that any Killing tensor is determined by its values in a point and the values of its first \emph{and second} derivatives in that point, cf.\ Section~\ref{sec:prolongation.Killing.tensors}.}
			with
		\item $T_{ijk}(x_0)$ given by \eqref{eq:BC2T} with $B_{ijk}(x_0)=\Psi$ and $C(x_0)=0$ and
		\item $T_{ijk,l}(x_0)$ given by \eqref{eq:prolongation:T}.
	\end{itemize}
	This map is linear in $S$ and polynomial in $\Psi$ and hence defines a
	regular map $f:Y\to G_{n_2}(\mathcal K(M))$.  Corollary~\ref{cor:Master:B}
	assures that its image lies in $\Phi_2(\mathcal A)$.

	The inverse of $f$ is a regular map $\Phi_2(\mathcal A)\to Y$,
	constructed as follows.  Consider a point $U_2\in\Phi_2(\mathcal A)$,
	i.e.~a subspace $U_2\subset\mathcal K(M)$ spanned by $n_2$ linearly
	independent Killing tensors $K^{(\alpha)}$ of an abundant system.
	Contracting $(i,j)$ in \eqref{eq:prolongation:K}, one obtains
	\begin{equation}
		\label{eq:KS=trK'}
		K^{ab}\tilde T_{abk}=K\indices{^a_{a,k}}
	\end{equation}
	with
	\[
		\tilde T_{ijk}
		=\frac13\young(ij)\bigl(g_{kj}t_i-T_{kji}\bigr).
	\]
	The definition of $\tilde T_{ijk}$ implies
	\begin{align*}
		g^{ij}\tilde T_{ijk}&=0&
		s_i&:=g^{jk}\tilde T_{ijk}=\frac n3t_k.
	\end{align*}
	Using \eqref{eq:SIC:V:linear}, we can express $T_{ijk}$ linearly in
	$\tilde T_{ijk}$ as
	\[
		T_{kji}
		=\frac32
		\left(
			\frac{n-2}{n(n-1)}s_ig_{jk}
			+\frac1{n-1}s_jg_{ik}
			-\tilde T_{ijk}
		\right).
	\]
	Taking Equation~\eqref{eq:KS=trK'} for each of the $K^{(\alpha)}$, we
	obtain a linear system of the form
	\[
		AX=B,
	\]
	where
	\begin{itemize}
		\item
			$A$ is the $n_2\!\times\!n_2$\,-matrix
			whose rows contain the components $K_{ij}^{(\alpha)}$,
		\item
			$B$ is the $n_2\!\times\!n$\,-matrix
			whose rows contain the components $g^{ab}K^{(\alpha)}_{ab,k}$ and
		\item
			$X$ the $n_2\!\times\!n$\,-matrix
			containing the components $\tilde T_{ijk}$.
	\end{itemize}
	The matrix $A$ is invertible.  Indeed, if the Killing tensors
	$K^{(\alpha)}$ are linearly dependent at some point $x$, then they are
	linearly dependent in $\mathcal K(M)$ by Proposition~\ref{prop:SIC:K},
	which is a contradiction.  Therefore the components $\tilde T_{ijk}$, given by
	\[
		X=A^{-1}B=\frac{(\operatorname{Adj} A)B}{\det A},
	\]
	are rational in $K^{(\alpha)}_{ij}$ and $K^{(\alpha)}_{ij,k}$.  Since the
	Codazzi tensor $B_{ijk}$ is linear in $T_{ijk}$, which in turn is linear
	in $\tilde T_{ijk}$, the same is true for $B_{ijk}$.  Note that derivatives and
	evaluation are linear operations, so that $K^{(\alpha)}_{ij}(x_0)$ and
	$K^{(\alpha)}_{ij,k}(x_0)$ are linear functions on $\mathcal K(M)$.
	Consequently, the components $B_{ijk}(x_0)$ are rational functions on
	$\mathcal K(M)^{n_2}$ with non-vanishing denominator on the Stiefel
	manifold $V_{n_2}\bigl(\mathcal K(M)\bigr)$.  Since the construction is
	independent of the choice of the basis $K^{(\alpha)}$ in the subspace
	$U_2$, these functions descend to regular functions on the Grassmannian
	$G_{n_2}(\mathcal K(M))$.  This defines a regular map $G_{n_2}(\mathcal
	K(M))\to\Sym^3T_{x_0}M$, mapping the subspace spanned by the $n_2$ Killing
	tensors $K^{(\alpha)}$ to the cubic $\Psi_{ijk}=\frac13B_{ijk}(x_0)$.  By
	Corollary~\ref{cor:Master:B} it restricts to a map $\Phi_2(\mathcal
	A)\to Y$.
\end{proof}

\section{Examples}
\label{sec:examples}

\begin{table}[ht]
\centering
\textbf{Families of non-degenerate second-order superintegable systems\newline in arbitrary dimension~$n\geq3$}
\bigskip

\begin{tabular}{p{5.7cm}|p{6.2cm}}
	\toprule
	\emph{\bfseries Euclidean geometry}\newline
	$g=\sum_i dx_i^2$
	& \emph{\bfseries Complex $n$-sphere}\newline
	  $g = -n(n-1)\,\sum_{i=1}^n \frac{dy_i^2}{y_n^2} = \sum_{i=0}^n dx_i^2$\newline
	  (ambient coordinates $x_i$ on $\mathbb{R}^{n+1}\supset\mathbb{S}^n$)
	\\
	\toprule
	\textbf{Isotropic harmonic oscillator}\newline
	$V=\omega^2\,\sum_ix_i^2 +\sum_i \alpha_ix_i$\newline
	$B=0$ mod gauge\newline
	$\bar t = 0$
	& ---
	\\ \midrule
	\textbf{Smorodinsky-Winternitz I}
	& \textbf{Smorodinsky-Winternitz I'} \\
	$V=\sum_{i=1}^n \left( \frac{a_i}{x_i^2}+\omega^2\,x_i^2 \right)$
	& $V=-\frac{y_n^2}{n(n-1)}\sum_{i=1}^n \biggl(\frac{a_i}{y_i^2}+\omega^2y_i^2\biggr)$	\\
	$B=-\frac{3}{n-1}\,\sum_i(x_i^2\ln x_i)$ mod gauge
	& $B=\frac{3n(n-1)}{2y_n^2}\sum_{i=1}^n(y_i^2\ln y_i)$ mod gauge\\
	$\bar t = -\frac3{n+2}\,\sum_k \ln x_k$
	& $\bar t = \frac3{n+2}\,\left((n+1)\ln x_n -\sum_{i\ne n} \ln x_i\right)$
	\\ \midrule
	\textbf{Smorodinsky-Winternitz II}
	& \textbf{Smorodinsky-Winternitz II'} \\
	$V=\sum_{i\in J}(4\omega^2x_i^2+a_ix_i)$\newline
	\hspace*{\fill}$+\sum_{i\in J^c}\biggl(\frac{a_i}{x_i^2}+\omega^2x_i^2\biggr)$
	& $V=-\frac{\,y_n^2}{n(n-1)}
	\Biggl[
	\sum_{i\in J}\Bigl(4\omega^2y_i^2+a_iy_i\Bigr)$\vspace{-1em}\newline
	\hspace*{\fill} $+\sum_{i\in J^c}\biggl(\omega^2y_i^2+\frac{a_i}{y_i^2}\biggr)
	\Biggr]$ \\
	$B=-\frac3{n-1}\sum_{i\in J^c}x_i^2\ln x_i$ mod gauge
	& $B=\frac{3n(n-1)}{2y_n^2}\sum_{i\in J^c}y_i^2\ln y_i$ mod gauge \\
	$\bar t=-\frac3{n+2}\,\sum_{i\in J^c}\ln x_i$
	& $\bar t=\frac3{n+2}\,\left((n+2)\ln y_n -\sum_{i\in J^c}\ln y_i\right)$
	\\ \midrule
	---
	& \textbf{Generic system on the $n$-sphere}\newline
	$V=\sum_{i=0}^n\frac{a_i}{x_i^2}$\newline
	$B=-\frac32\sum_{i=0}^nx_i^2\ln x_i$ mod gauge\newline
	$\bar t=-\frac3{n+2}\,\sum_{i=0}^n \ln(x_i)$
	\\ \bottomrule
\end{tabular}
\bigskip
\caption{%
	Overview of examples discussed in the text. Examples in the same row are Stäckel equivalent. Note that the potential $V$ is stated up to an additive constant, and $B$ up to gauge terms. All examples are abundant with gauge choice $C=0$.
}
\label{tab:examples}
\end{table}

Arbitrary-dimensional families of superintegrable systems can be obtained by generalising low-dimensional examples, and subsequently through Bôcher
contractions and Stäckel transforms. In the following we list some
important families of superintegrable systems in arbitrary dimension, cf.\ Table~\ref{tab:examples}.
For each system we detail its potential $V$ (omitting the additive constant) and its structure function $B$ in a gauge where $C=0$ (all
examples in the list are abundant in particular). 

On flat space we use standard coordinates $x_i$.
On the $n$-sphere, in order to make the similarity
between the flat and the curved case apparent, we use local coordinates $y_i$ with
\[
	g=-n(n-1)\sum_{i=1}^n\frac{dy_i^2}{y_n^2},
\]
except for the generic system, which we write in the coordinates $x_i$ of the ambient Euclidean space.

\subsection{Isotropic harmonic oscillator.}

\begin{align*}
	V&=\sum_{i=1}^n(\omega^2x_i^2+a_ix_i)&
	B&=0
\end{align*}
This is the trivial case already mentioned in Example~\ref{ex:IHO}. Its structure tensor vanishes and therefore, up to gauge terms, all its structure functions vanish as well.

\subsection{Smorodinsky-Winternitz I}

\begin{align*}
	V&=\sum_{i=1}^n\biggl(\frac{a_i}{x_i^2}+\omega^2x_i^2\biggr)&
	B&=-\frac{3}{n-1}\sum_{i=1}^nx_i^2\ln x_i
\end{align*}
This system was first described on flat space by Friš et al.~\cite{FMSUW65}
and is often called the ``generic system on flat space''.  For low dimensions,
other labels have also been used in the literature such as $(0,11,0)$ in
\cite{Kress&Schoebel} or [E1] in \cite{Kalnins&Kress&Pogosyan&Miller}.  The
corresponding label in three dimensions is [I] in \cite{KKM06a,KKM07c,Capel}.

This system has a Stäckel invariant counterpart on the $n$-sphere, known as
[I'] in dimension three \cite{KKM07c}:
\begin{align*}
	V&=-\frac{y_n^2}{n(n-1)}\sum_{i=1}^n\biggl(\frac{a_i}{y_i^2}+\omega^2y_i^2\biggr)&
	B&=\frac{3n(n-1)}{2y_n^2}\sum_{i=1}^ny_i^2\ln y_i
\end{align*}

\subsection{Smorodinsky-Winternitz II}

\begin{align*}
	V&
	=\sum_{i\in J}(4\omega^2x_i^2+a_ix_i)
	+\sum_{i\in J^c}\biggl(\frac{a_i}{x_i^2}+\omega^2x_i^2\biggr)&
	B&=-\frac3{n-1}\sum_{i\in J^c}x_i^2\ln x_i
\end{align*}
In dimension~$2$, this system is also denoted by $(0,1,0)$ \cite{Kress&Schoebel} or [E2] in \cite{Kalnins&Kress&Pogosyan&Miller}.
In dimension~$3$, it is labeled [IV] \cite{KKM06a}.
This system formally appears as a superposition of the Smorodinski-Winternitz
system I and the isotropic harmonic oscillator, defined by a partition of the
index set,
\[
	\{1,2,\ldots,n\}=J\cup J^c.
\]
This might be an indication for the existence of a composition of
superintegrable systems similarly to the operad construction of separable
systems on spheres \cite{Schoebel&Veselov}.

The corresponding system on the $n$-sphere reads
\begin{align*}
	V&=-\frac{\,y_n^2}{n(n-1)}
	\Biggl[
		 \sum_{i\in J}\Bigl(4\omega^2y_i^2+a_iy_i\Bigr)
		+\sum_{i\in J^c}\biggl(\omega^2y_i^2+\frac{a_i}{y_i^2}\biggr)
	\Biggr] \\
	B&=\frac{3n(n-1)}{2y_n^2}\sum_{i\in J^c}y_i^2\ln y_i
\end{align*}
In low dimensions, the labels [S1] \cite{Kalnins&Kress&Pogosyan&Miller} and [IV'] \cite{KKM06a,Capel} have been used for this system.

\subsection{The generic system on the $n$-sphere}

\begin{align*}
	V&=\sum_{i=0}^n\frac{a_i}{x_i^2}&
	B&=-\frac32\sum_{i=0}^nx_i^2\ln x_i
\end{align*}
This system on the $n$-sphere, not to be confused with the generic system on
flat space, is labelled [S9] for 2D in \cite{Kalnins&Kress&Pogosyan&Miller} and
[VIII] in \cite{KKM07c,KKM06a} for 3D. Its name and its significance result from the fact that in
dimensions two and three any other non-degenerate second-order superintegrable
system can be obtained from it via Stäckel transforms and contractions.

Note that the term ``generic'' has a precise meaning in our algebraic
geometric context:  it refers to a Zariski open subset in the classification
space.  We remark that if the latter turns out to be reducible, there might as
well be several ``generic'' systems, as observed for the Euclidean plane
\cite{Kress&Schoebel}.

\sloppy
\bibliographystyle{amsalpha}
\bibliography{Ingrid}
\bigskip

\end{document}